\documentclass[12pt,oneside,letterpaper]{amsart}
\usepackage{amssymb}
\usepackage{epsfig}
\usepackage{subfigure}
\usepackage[left=1in,top=1in,right=1in,bottom=1in]{geometry}
\newcommand{\dup}[2]{\langle#1,#2 \rangle}

\newcommand{\ip}[2]{(#1,#2)}
\newcommand{\Co}{\mathrm{Co}} 
\newcommand{\cdual}{{\ast}}
\newcommand{\supp}{\mathrm{supp}}
\renewcommand{\theenumi}{\alph{enumi}} 
\newtheorem{theorem}{Theorem}[section]
\newtheorem{lemma}[theorem]{Lemma}
\newtheorem{assumption}[theorem]{Assumption}
\newtheorem{corollary}[theorem]{Corollary}
\newtheorem{proposition}[theorem]{Proposition}
 \theoremstyle{definition}
\newtheorem{definition}[theorem]{Definition}

\newtheorem{remark}[theorem]{Remark}
\newcommand{\SL}{\mathrm{SL}_2(\mathbb{R})}

\begin{document}

\title{Coorbit Spaces for Dual Pairs} 

\author{Jens Gerlach Christensen} 
\address{2307 Mathematics Building, Department of mathematics, 
  University of Maryland, College Park} 
\email{jens@math.umd.edu}
\urladdr{http://www.math.umd.edu/\textasciitilde jens}
\thanks{The first author gratefully acknowledges support from the
  Louisiana Board of Regents under grant LEQSF(2005-2007)-ENH-TR-21
  and NSF grant DMS-0801010}

\author{Gestur \'Olafsson}
\address{322 Lockett Hall, Department of mathematics, 
  Louisiana State University} 
\email{olafsson@math.lsu.edu}
\urladdr{http://www.math.lsu.edu/\textasciitilde olafsson}
\thanks{The research of the second author was supported by
  NSF grants DMS-0402068 and DMS-0801010}

\begin{abstract}
In this paper we present an abstract framework for construction
of Banach spaces of distributions from group representations.
This generalizes the theory of coorbit spaces
initiated by H.G. Feichtinger and K. Gr\"ochenig
in the 1980's. Spaces that can be described by this new technique
include the whole Banach-scale of Bergman spaces
on the unit disc. For these Bergman 
spaces we show that atomic decompositions can be constructed through
sampling.%, even though the representations involved are not integrable.
We further present a wavelet characterization of Besov spaces
on the forward light cone.
\end{abstract}

\maketitle
\begin{center} {\today}
\end{center}

\section{Introduction}
\label{sec:intro}
In the 1980's H.G. Feichtinger and K. Gr\"ochenig 
presented a unified framework for generation of
Banach spaces of distributions using group representations.
Their results published in 
\cite{Feichtinger1988,Feichtinger1989a,Feichtinger1989b} and
\cite{Grochenig2001}
are based on a unitary irreducible representation $(\pi,H)$ 
of a locally compact group $G$ with left-invariant Haar measure
$dx$. 
The construction of Feichtinger and Gr\"ochenig 
requires that the space
of analyzing vectors
\begin{equation*}
  \mathcal{A}_w = \left\{u\in H \,\Big|\, \int_G |(\pi(x)u,u)|w(x)\,dx < \infty \right\}
\end{equation*}
is non-zero for a submultiplicative weight $w:G\mapsto\mathbb{R}^+$.
Here $(u,v)$ is the inner product of  $u,v\in H$. 
For a non-zero analyzing vector $u$ the space
\begin{equation*}
  H_w^1 = \left\{v\in H \,\Big|\, \int_G |(\pi(x)u,v)|w(x)\,dx < \infty \right\}
\end{equation*}
is a Banach space, which does not depend on the chosen $u\in \mathcal{A}_w$.
Denote by $(H_w^1)^*$ 
the conjugate dual of $H_w^1$. For a left-invariant Banach
function space $Y$ on $G$ define
\begin{equation*}
  \Co Y = \{ v\in (H_w^1)^* \,|\, 
  \text{$(x\mapsto \dup{v}{\pi(x)u})$ is in $Y$}  \}.
\end{equation*}
Feichtinger and Gr\"ochenig show, that
$\Co Y$ is a $\pi$-invariant Banach space of distributions which
is isometrically isomorphic to a reproducing kernel Banach subspace
of $Y$. Further they construct atomic decompositions and frames
for these spaces in the case where the analyzing vector
$u$ is chosen such that $W_u(u)$ is in a certain Wiener amalgam space
(see for example Lemma 4.6(i) in \cite{Grochenig1991}).

In the article \cite{Christensen2009} 
the authors gave examples of coorbits for which the space
$\mathcal{A}_w$ is the zero space, yet both the construction
of $\Co Y$ and atomic decompositions yield non-trivial results.
The present article proposes a generalized coorbit theory, which
is able to account for the examples from \cite{Christensen2009}.
The idea is to replace the space $H_w^1$ with a Fr\'echet
space $S$. For square integrable representations of
Lie groups the space of smooth vectors is a natural choice.
As an example, the smooth vectors of the discrete series representation
of the group
\begin{equation*}
  G = \Big\{ 
  \begin{pmatrix}
    a & b \\ 0 & a^{-1}
  \end{pmatrix}\, \Big|\, a>0,b\in\mathbb{R}
  \Big\}
  \subseteq \mathrm{SL}_2(\mathbb{R})
\end{equation*}
are used to give a complete wavelet characterization of the 
Bergman spaces of holomorphic functions on the unit disc.
We further present a wavelet characterization of the
Besov spaces on the forward light cone as defined in 
\cite{Bekolle2004}. This example can be described by
the theory of Feichtinger and Gr\"ochenig, however we
include it here as it is of interest in its own right.
We expect that this construction will generalize to other
symmetric cones.

\section{Coorbit Spaces for Dual Pairs}
\label{sec:coorbit-spaces-dual}
\noindent

Let $S$ be a Fr\'echet space and let $S^\cdual$ be the 
conjugate linear dual equipped with the weak topology. 
We assume that $S$ is
continuously imbedded and 
weakly dense in $S^\cdual$. The conjugate
dual pairing of elements $v\in S$ and $\phi\in S^\cdual$ will be denoted
by $\dup{\phi}{v}$. Let $G$ be a locally compact 
group with a fixed left Haar measure
$dx$, and assume that $(\pi,S)$ is a continuous representation of
$G$, i.e. $g\mapsto \pi(g)v$
is continuous for all $v\in S$.  As usual, define the contragradient
representation $(\pi^\cdual,S^\cdual)$ by
\begin{equation*}
  \dup{\pi^\cdual(x)\phi}{v}=\dup{\phi}{\pi(x^{-1})v}.
\end{equation*}
Then $\pi^*$ is a continuous representation of $G$ on $S^\cdual$. 
For a fixed vector 
$u\in S$ define the linear map $W_u:S^*\to C(G)$ by
\begin{equation*}
  W_u(\phi)(x) = \dup{\phi}{\pi(x)u} = \dup{\pi^*(x^{-1})\phi}{u}.
\end{equation*}
The map $W_u$ is called \emph{the voice transform} or 
\emph{the wavelet transform}.
If $F$ is a function on $G$ then define the left translation of $F$
by an element $x\in G$ as
\begin{equation*}
  \ell_x F(y) = F(x^{-1}y)
\end{equation*}
If $Y$ is a space of functions, then $Y$ is called 
invariant if $\ell_x F\in Y$ when $F\in Y$.
In the following we will always assume that 
the space $Y$ of functions on $G$ is a Banach space
for which left translation is continuous and convergence 
implies convergence (locally) in Haar measure on $G$.
Examples of such spaces are $L^p(G)$ for $1\leq p \leq \infty$ 
and any space continuously included in an $L^p(G)$.

\begin{assumption}\label{assumption1}
  Assume that there is a non-zero cyclic vector $u\in S$ 
  satisfying the following properties
  \begin{enumerate}
    \renewcommand{\labelenumi}{(R\arabic{enumi})}
    \renewcommand{\theenumi}{(R\arabic{enumi})}
  \item the reproducing formula $W_{u}(v)*W_{u}(u)=W_{u}(v)$ is true
    for all $v\in S$ \label{r1}
  \item 
    the mapping
    $Y\ni F\mapsto \int_G F(x)W_u(u)(x^{-1})\,dx \in \mathbb{C}$
    is continuous \label{r2}
  \item if $F=F*W_u(u)\in Y$ then the mapping
    $S \ni v \mapsto \int F(x) \dup{\pi^\cdual(x)u}{v} \,dx
    \in\mathbb{C}$ 
    is in $S^\cdual$\label{r3}
  \item the mapping 
    $S^\cdual\ni \phi \mapsto \int \dup{\phi}{\pi(x)u} 
    \dup{\pi^\cdual(x)u}{u}  \,dx \in\mathbb{C}$ 
    is weakly continuous \label{r4}
  \end{enumerate}
\end{assumption}
A vector $u$ satisfying Assumption~\ref{assumption1} 
is called an \emph{analyzing
vector}. 
Note that assumption \ref{r2} and the left invariance of
$Y$ ensure that the convolution
\begin{equation*}
  F*W_u(u)(y) = \int_G F(x)W_u(u)(x^{-1}y) \,dx
\end{equation*}
is well-defined for all $F\in Y$ at every point $y \in G$.
If $F=F*W_u(u) \in Y$ then \ref{r3} implies the existence of a
unique $\phi\in S^*$ such that
\begin{equation*}
  \dup{\phi}{v}
  = \int_G F(x)\dup{\pi^*(x)u}{v} \,dx.
\end{equation*}
This $\phi$ is denoted 
$\pi^*(F)u$.
Also \ref{r4} implies that there is an element
$v\in S$ such that
\begin{equation*}
  \dup{\phi}{v} = \int \dup{\phi}{\pi(x)u}\dup{\pi^\cdual(x)u}{u} \,dx
\end{equation*}
for all $\phi\in S^\cdual$. 
This ensures that the vector $v\in S$ can
be weakly defined by
\begin{equation*}
  v = \pi(W_u(u)^\vee)u = \int_G W_u(u)^\vee(x)\pi(x)u\, dx
\end{equation*}
where we have used the notation 
$f^\vee(x) = f(x^{-1})$.

Define the subspace $Y_u$ of $Y$ by 
$$Y_u = \{ F\in Y\, |\, F=F*W_u(u) \},$$ then
the following result holds 
\begin{lemma}
  If $F$ and $u$ satisfy \ref{r2}, then the space $Y_u$ is closed and
  hence a reproducing kernel Banach space.
\end{lemma}

\begin{proof}
  Let $\{ F_n\}$ be a sequence in $Y_u$ which converges to
  $F\in Y$. Then, since we assumed that convergence in
  $Y$ implies convergence in measure, we know
  that there is a subsequence $F_{n_k}$ which converges to
  $F$ almost everywhere. $F*W_u(u)(y)$ is defined for all
  $y$ by assumption \ref{r2} and we see that for a fixed 
  $y$ outside
  a null-set
  \begin{align*}
    | F(y) - F*W_u(u)(y)|
    &\leq | F(y) - F_{n_k}(y)| \\ &\qquad +
    |F_{n_k}(y)- F_{n_k}*W_u(u)(y)| \\ &\qquad +
    |F_{n_k}*W_u(u)(y)- F*W_u(u)(y)|
  \end{align*}
  The first term can be made arbitrarily small and the second
  term is zero. The last term can be estimated by
  \begin{equation*}
    C\|\ell_{y^{-1}} F_{n_k} - \ell_{y^{-1}} F \|_Y
  \end{equation*}
  by assumption \ref{r2} and the left invariance of $Y$ ensures
  that it can be made arbitrarily small (using that 
  $F_{n_k}$ converges to
  $F$ in norm). Therefore $F=F*V_u(u)$ almost everywhere and
  $F\in Y_u$.
\end{proof}

Define the space
\begin{equation}
  \Co_S^u Y = \{ \phi \in S^\cdual \,|\, W_u(\phi) \in Y \}
\end{equation}
equipped with the norm $\| \phi \| = \| W_u(\phi)\|_Y$. 
The space $\Co_S^u Y$ is called the coorbit space of $Y$ related
to $u$ and $S$.

\begin{theorem} \label{mainthm} Assume that $Y$ and $u$ satisfy
  Assumption~\ref{assumption1}, then
  %Define the space
  %\begin{equation}
  %  \Co_S^u Y = \{ \phi \in S^\cdual | W_u(\phi) \in Y \}
  %\end{equation}
  %equipped with the norm $\| \phi \| = \| W_u(\phi)\|_Y$. 
  %The space $\Co_S^u Y$ is called the coorbit space of $Y$ related
  %to $u$ and $S$.
  %The following properties hold
  \begin{enumerate}
    \renewcommand{\labelenumi}{(\alph{enumi})}
    \renewcommand{\theenumi}{\alph{enumi}}
  \item $W_u(v)*W_u(u) = W_u(v)$ for $v\in \Co^u_S Y$. \label{prop1}
  \item The space $\Co_S^u Y$ is a $\pi^\cdual$-invariant Banach
    space. \label{prop2}
  \item $W_u:\Co_S^u Y\to Y$ intertwines $\pi^\cdual$ and left
    translation \label{prop3}
  \item $\Co_S^u Y = \{ \pi^\cdual(F)u | F\in
    Y_u\}$. \label{prop5}
  \item $W_u:\Co_S^u Y \to Y_u$ is an isometric
    isomorphism\label{prop6}
  \end{enumerate}
\end{theorem}

\begin{proof}
  % Let us first show that the space $Y_u$ is a Banach space.
  % It suffices to show that $Y_u$ is closed, so let
  % $\{ F_n\}$ be a sequence in $Y_u$ which converges to
  % $F\in Y$. Then, since we assumed that convergence in
  % $Y$ implies convergence in measure, we know
  % that there is a subsequence $F_{n_k}$ which converges to
  % $F$ almost everywhere. $F*W_u(u)(y)$ is defined for all
  % $y$ by assumption \ref{r2} and we see that for a fixed 
  % $y$ outside
  % a null-set
  % \begin{align*}
  %   | F(y) - F*W_u(u)(y)|
  %   &\leq | F(y) - F_{n_k}(y)| \\ &\qquad +
  %   |F_{n_k}(y)- F_{n_k}*W_u(u)(y)| \\ &\qquad +
  %   |F_{n_k}*W_u(u)(y)- F_{n}*W_u(u)(y)|
  % \end{align*}
  % The first term can be made arbitrarily small and the second
  % term is zero. The last term can be estimated by
  % \begin{equation*}
  %   C\|\ell_{y^{-1}} F_{n_k} - \ell_{y^{-1}} F_{n} \|_Y
  % \end{equation*}
  % by assumption \ref{r2} and the left invariance of $Y$ ensures
  % that it can be made arbitrarily small (using that 
  % $F_{n_k}$ converges to
  % $F$ in norm). Therefore $F=F*V_u(u)$ almost everywhere and
  % $F\in Y_u$.

  (\ref{prop1}) We show that the reproducing formula holds
  for all $\phi\in S^*$. 
  The space $S$ is weakly dense in $S^\cdual$, so
  choose a net $v_\alpha$ in $S$ for which $v_\alpha \to
  \phi$ weakly in $S^*$. By assumption \ref{r1} the
  reproducing formula $W_u(v_\alpha)*W_u(u) = W_u(v_\alpha)$ 
  holds for each $v_\alpha$.  The continuity requirement \ref{r4}
  gives that
  \begin{align*}
    \phi\mapsto 
    &\int \dup{\pi^\cdual(y^{-1})\phi}{\pi(x)u} \dup{\pi^\cdual(x)u}{u} \,dx\\
    &=\int \dup{\phi}{\pi(x)u} \dup{u}{\pi(x^{-1}y)u} \,dx \\
    &= W_u(\phi)*W_u(u)(y)
  \end{align*}
  is weakly continuous. Therefore $W_u(v_\alpha)*W_u(u)(y) \to
  W_u(\phi)*W_u(u)(y)$ for every $y\in G$.  By assumption
  $W_u(v_\alpha)(y)\to W_u(\phi)(y)$ for all $y\in G$, and we conclude
  that 
  \begin{equation*}
    W_u(\phi)(y) = W_u(\phi)*W_u(u)(y)\qquad \text{for all $y\in G$.}     
  \end{equation*}
  This reproducing formula is valid for all $\phi\in S^*$ and
  hence also for $\phi\in \Co_S^uY  \subseteq S^\cdual$.

  (\ref{prop2},\ref{prop3})
  We now show that $\| \phi\| = \| W_u(\phi)\|_Y$ is indeed a norm. The
  only non-obvious property is that $\| \phi\|=0$ implies $\phi=0$. If 
  $\| \phi\|=0$ then $\| W_u(\phi)\|_Y =0$ and so $\dup{\phi}{\pi(x)u}=0$ for
  almost all $x$.  The function $x\mapsto \dup{\phi}{\pi(x)u}$ is
  continuous and thus it is identically zero for all $x$. But
  $u$ is cyclic in $S$, so $\dup{\phi}{v}=0$ for all
  $v\in S$. Thus $\phi=0$.  This also proves the
  injectivity of $W_u$.

  Next we prove that the space $\Co_S^u Y$ is complete.
  Assume that $v_n$ is a Cauchy sequence
  in $\Co_S^u Y$. Then $W_u(v_n)$ is a Cauchy sequence in $Y_u$ and
  $W_u(v_n)$ converges to a function $F\in Y_u$. 
  Assumption \ref{r3} implies that 
  $\phi\in S^\cdual$ defined by
  \begin{equation*}
    \dup{\phi}{v} = \int F(x)\dup{\pi^\cdual(x)u}{v} \,dx
  \end{equation*}
  is in $S^*$,
  and it follows that
  \begin{align*}
    W_u(\phi)(y)
    &= \dup{\phi}{\pi(y)u} \\
    &= \int F(x)\dup{\pi^\cdual(x)u}{\pi(y)u} \,dx \\
    &= \int F(x)\dup{u}{\pi(x^{-1}y)u} \,dx \\
    &= F*W_u(u)(y) \\
    &= F(y).
  \end{align*}
  Thus $\phi\in \Co_S^u Y$.

  The definition of $\pi^*$ and the left invariance of $Y$ 
  ensure that $\Co_S^u Y$ is $\pi^*$-invariant and
  that $W_u$ intertwines $\pi^*$ and left translation:
  Assume that $\phi$ is in $\Co_S^u Y$, then the voice transform of
  $\pi^\cdual(y)\phi$ is
  \begin{equation*}
    W_u(\pi^\cdual(y)\phi)(x)
     = \dup{\pi^\cdual(y)\phi}{\pi(x)u} 
     = \dup{\phi}{\pi(y^{-1}x)u}
     = \ell_yW_u(\phi)(x).
   \end{equation*}

  (\ref{prop6}) We now show that $W_u(\Co_S^u Y)=Y_u$. 
  If $\phi\in \Co_S^uY$ then $W_u(\phi)\in Y$ and also
  $W_u(\phi)=W_u(\phi)*W_u(u) \in Y_u$.  If on the other hand $F\in
  Y_u$ then $F=F*W_u(u)$ and assumption \ref{r3} again tells us
  that there is a $\phi\in S^\cdual$ defined by
  \begin{equation*}
    \dup{\phi}{v} = \int F(x)\dup{\pi^\cdual(x)u}{v} \,dx
  \end{equation*}
  for $v\in S$. Direct calculation shows that
  \begin{equation*}
    W_u(\phi) = F*W_u(u) =F \in Y
  \end{equation*}
  such that $\phi\in \Co_S^uY$. Therefore $W_u:\Co_S^u Y \to Y_u$
  is surjective.  That $W_u$ is an isometry follows directly from the
  definition of the norm.

  (\ref{prop5})  Above we have shown that for $F\in Y_u$ there is a
  $\phi\in \Co_S^u Y$ such that $\phi=\pi(F)u$. If on the other
  hand $\phi\in \Co_S^u Y$ then let $f=W_u(\phi)=f*W_u(u) \in
  Y*W_u(u)$. Then by \ref{r3} $\pi^\cdual(f)u$ defines an element in
  $S^\cdual$ and
  \begin{align*}
    \dup{\pi^\cdual(F)u}{\pi(y)u}
    &= \int F(x) \dup{\pi^\cdual(x)u}{\pi(y)u} \,dx\\
    &= F*W_u(u)(y)\\
    &= F(y) \\
    &= \dup{\phi}{\pi(y)u}.
  \end{align*}
  This shows that $\pi^\cdual(F)u$ and $\phi$ agree for all $\pi(y)u$,
  and since $u$ is cyclic in $S$, it follows that
  $\pi^\cdual(F)u=\pi^\cdual(F)u$ and $\phi$ are the same element
  in $S^\cdual$.
\end{proof}

\begin{theorem}\label{thm:subspace}
  If Assumption~\ref{assumption1} holds for $u$ and 
  a Banach function space $Y$ then it also holds for any
  quasi Banach space $\widetilde{Y}$ continuously included in $Y$.
  In particular 
  $\Co_S^u\widetilde{Y}$ is a well-defined quasi Banach space
  satisfying Theorem~\ref{mainthm} and
  $\Co_S^u\widetilde{Y}$ is continuously included in $\Co_S^u Y$.
\end{theorem}

% \begin{theorem}\label{thm:subspace}
%   If Assumption~\ref{assumption1} holds for 
%   a Banach function space $Y$ and $\widetilde{Y}$ is a
%   quasi Banach space continuously included in $Y$, then
%   $\Co_S^u\widetilde{Y}$ is a well-defined quasi Banach space
%   satisfying Theorem~\ref{mainthm}.
%   Furthermore $\Co_S^u\widetilde{Y}$ is continuously
%   included in $\Co_S^u Y$
% \end{theorem}

\begin{remark}
  \label{rem:proj}
  If we replace condition \ref{r2} by the assumption that 
  the mapping $Y\ni F\mapsto F*W_u(u) \in Y$ is continuous,
  then $Y_u = Y*W_u(u)$ and the convolution operator 
  $F\mapsto F*W_u(u)$ is a continuous
  projection onto the image of $W_u$.
  This is the version of the assumptions found in \cite{Christensen2009}
  and in \cite{Christensen2009a}.
  However we have opted for the more general assumption which only
  ensures the existence of the convolution. The reason for
  this is that we aim at giving a wavelet characterization of
  Bergman spaces related to symmetric cones,
  in which case the projection might not be 
  defined (see \cite{Bekolle2001,Bekolle2004}). 
  Further it is often easier to show that the function $W_\psi(\psi)$ 
  is in the dual of $Y$ rather than $Y*W_\psi(\psi) \subseteq Y$.
\end{remark}    

\begin{remark}
  Theorem 4.2(i) in \cite{Feichtinger1989a} states that $\Co_{FG} Y$ is
  continuously included in $(\mathcal{H}_m^1)^*$, and Theorem
  4.5.13(d) in \cite{Rauhut2005} states further that $\mathcal{H}_m^1$ is
  continuously included in $\Co_{FG} Y$.  It is an open problem
  whether similar statements are true for $S$, $\Co_S^u Y$ and $S^*$.
\end{remark}
    
The following theorem tells us which analyzing vectors will give the
same coorbit space.
\begin{theorem}[Dependence on the analyzing vector]
  \label{thm:indep-of-u}
  If $u_1$ and $u_2$ both satisfy Assumption~\ref{assumption1} and for
  $i,j\in\{1,2\}$ the following properties hold
  \begin{itemize}
  \item there are non-zero constants $c_{i,j}$ such that
    $W_{u_i}(v)*W_{u_j}(u_i) = c_{i,j}W_{u_j}(v)$     for all $v\in S$
  \item $Y_{u_i}\ni f\mapsto f*W_{u_j}(u_i)\in Y$ is continuous
  \item $S^\cdual\ni \phi \mapsto \int \dup{\phi}{\pi(x)u_i}
    \dup{\pi^\cdual(x)u_i}{u_j} \,dx \in\mathbb{C}$ is weakly
    continuous
  \end{itemize}
  then $\Co_S^{u_1} Y = \Co_S^{u_2}Y$ with equivalent norms.
\end{theorem}

\begin{proof}
  Assume that $u_1$ and $u_2$ are two analyzing vectors, i.e.  they
  satisfy the properties Assumption~\ref{assumption1}.  We claim first that
  \begin{equation*}
    W_{u_1}(v)*W_{u_2}(u_1) = c_{1,2}W_{u_2}(v)
  \end{equation*}
  for all $v\in S^\cdual$. With $v\in S$ this is true by the
  assumption.  The space $S$ is weakly dense in $S^\cdual$ and
  therefore the identity $W_{u_1}(v)*W_{u_2}(u_1) = c_{1,2}W_{u_2}(v)$ is
  true for all $v\in S^\cdual$. This is verified by applying the third
  continuity condition to the integral
  \begin{equation*}
    W_{u_1}(v)*W_{u_2}(u_1)(y)
    = \int \dup{\pi^\cdual(y^{-1})v}{\pi(x)u_1}\, \dup{\pi^\cdual(x)u_1}{u_2} \,dx.
  \end{equation*}
  
  If $W_{u_1}(v)\in Y$ then $W_{u_1}(v)\in Y_{u_1}$ and
  $W_{u_1}(v)*W_{u_2}(u_1)=c_{1,2}W_{u_2}(v)\in Y$
  by assumption. The continuity assumption gives the inequality
  $$\|W_{u_2}(v) \|_Y 
  = c_{1,2}^{-1} \|W_{u_1}(v)*W_{u_2}(u_1) \|_Y
  \leq C \| W_{u_1}(v)\|_Y. $$
  Symmetry   then gives us that $\Co_S^{u_1} Y = \Co_S^{u_2} Y$ with
  equivalent norms.
%
%   It remains to show that the norms $\| v\|_1=\| W_{u_1}(v)\|_Y$ and
%   $\| v\|_2=\| W_{u_2}(v)\|_Y$ are equivalent norms on $\Co_S^{u_1} Y
%   = \Co_S^{u_2} Y$.  We have assumed that the mappings $f\mapsto
%   f*W_{u_2}(u_1)$ and $f\mapsto f*W_{u_2}(u_1)$ are continuous.  This
%   means that $\| f*W_{u_1}(u_2)\|_Y \leq A_1 \| f\|_Y$ and $\|
%   f*W_{u_2}(u_1)\|_Y \leq A_2 \| f\|_Y$.  But then
%   \begin{equation*}
%     c_{2,1}\| v\|_1 
%     = c_{2,1}\| W_{u_1} (v)\|_Y
%     = \| W_{u_2} (v)*W_{u_1}(u_2)\|_Y
%     \leq A_1 \| W_{u_2} (v)\|_Y
%     = A_1 \| v\|_2
%   \end{equation*}
%   Similarly $c_{1,2}\| v\|_2 \leq A_2 \| v\|_1$ which shows
%  the norms are equivalent.
\end{proof}

In the following we will describe how the choice of the Fr\'echet
space $S$ affects the coorbit space. We will show that there is great
freedom when choosing $S$.

\begin{theorem}[Dependence on the Fr\'echet space] \label{thm:dependence-on-S}
  Let $S$ and $T$ be Fr\'echet spaces which are weakly dense in their
  conjugate duals $S^\cdual$ and $T^\cdual$ respectively.  Let $\pi$
  and $\widetilde{\pi}$ denote representations of $G$ on $S$ and $T$
  respectively. Assume there is a vector $u\in S$ and
  $\widetilde{u}\in T$ such that the requirements in 
  Assumption~\ref{assumption1}
  are satisfied by both $(u,S)$ and
  $(\widetilde{u},T)$. Also assume that the conjugate dual
  pairings of $S^\cdual\times S$ and $T^\cdual\times T$ satisfy
  $\dup{u}{\pi(x)u}_S
  =\dup{\widetilde{u}}{\widetilde{\pi}(x)\widetilde{u}}_T$ for all
  $x\in G $.  Then $\Co_S^u Y$ and $\Co_T^{\widetilde{u}} Y$ are
  isometrically isomorphic. 
  The isomorphism is given by $W_{\widetilde{u}}W_u^{-1}$.
\end{theorem}

\begin{proof}
  Let $W_u(\phi)(x) = \dup{\phi}{\pi(x)u}_S$ for $\phi \in \Co_u^S Y$ and
  $W_{\widetilde{u}}(\widetilde{v}')(x) =
  \dup{\widetilde{v}'}{\widetilde{\pi}(x)\widetilde{u}}_T$ for
  $\widetilde{v}' \in \Co_{\widetilde{u}}^T Y$.  Since it is assumed
  that $W_u(\pi(x)u)=W_{\widetilde{u}}(\pi(x)\widetilde{u})$ for all
  $x\in G$ the spaces $\Co_u^S Y$ and $\Co_{\widetilde{u}}^T Y$ are
  both isometrically isomorphic to the space
  $Y_u=Y_{\widetilde{u}}$.  The isomorphism
  between $\Co_S^u Y$ and $\Co_T^{\widetilde{u}} Y$ is exactly
  $W_{\widetilde{u}}^{-1} W_u: \Co_u^S Y \to \Co_{\widetilde{u}}^T Y$.
\end{proof}

Let $\pi$ be a unitary irreducible representation of $G$ on
$\mathcal{H}$.  Assume that the Fr\'echet spaces $S$ and $T$ are
$\pi$-invariant and that $(S,\mathcal{H},S^\cdual)$ and
$(T,\mathcal{H},T^\cdual)$ are Gelfand triples with the common Hilbert
space $\mathcal{H}$.  Then $S\cap T$ is $\pi$-invariant. If we can
choose a non-zero vector $u\in S\cap T$, such that $u$ is analyzing for
both $S$ and $T$, then
\begin{equation*}
  \dup{u}{\pi(x)u}_S = (u,\pi(x)u)_{\mathcal{H}} = \dup{u}{\pi(x)u}_T
\end{equation*}
and we are in the situation of the previous theorem.  We summarize the
statement as
\begin{corollary}
  Assume that $(S,\mathcal{H},S^\cdual)$ and
  $(T,\mathcal{H},T^\cdual)$ are Gelfand tripples and assume there is
  an analyzing vector $u\in S\cap T$ such that both $(u,S)$ and
  $(u,T)$ satisfy Assumption~\ref{assumption1} for some Banach space
  $Y$, then $\Co_S^u Y$ and $\Co_T^{u} Y$ are isometrically isomorphic.
\end{corollary}

If the Fr\'echet space $S$ is continuously included and dense in
the Fr\'echet space $T$, then we can regard the space
$T^\cdual$ as a subspace of $S^\cdual$.  With this identification the
two coorbit spaces will be equal. We state the following
\begin{theorem}
  Let $(\pi,\mathcal{H})$ be a unitary representation of
  $G$, and let $(S,\mathcal{H},S^\cdual)$ and
  $(T,\mathcal{H},T^\cdual)$ be Gelfand triples for which $(\pi,S)$
  and $(\pi,T)$ are representations of $G$.  Assume that $i:S\to T$ is
  a continuous linear inclusion and that there is $u\in S$ such that
  both $(u,S)$ and $(i(u),T)$ satisfy
  Assumption~\ref{assumption1}.  Then the map $i^\cdual$ restricted to
  $\Co_T^{i(u)} Y$ is an isometric isomorphism between $\Co_T^{i(u)}
  Y$ and $\Co_S^u Y$.
\end{theorem}

\begin{proof}
  Since the vector $i(u)$ is assumed cyclic in $T$, we see that 
  $i(S)$ is dense in $T$, and therefore $i^*:T^*\to S^*$ is
  injective. This allows us to view $T^*$ as a subspace of
  $S^*$.
  
  Let $W_u(\phi)(x) = \dup{\phi}{\pi(x)u}_S$ and $W_{i(u)}(\widetilde{v}')
  = \dup{\widetilde{v}'}{\pi(x)i(u)}_T$.  For $\phi \in \Co_S^u Y$ we have
  \begin{align}
    W_u(\phi)*W_{i(u)}(i(u)) (x) \label{eq:5}
    &= \int W_u(\phi)(y) (i(u),\pi(y^{-1}x)i(u))_\mathcal{H} \,dy \\
    &= \int W_u(\phi)(y) (u,\pi(y^{-1}x)u)_\mathcal{H} \,dy \notag\\
    &= W_u(\phi)*W_u(u) \notag\\
    &= W_u(\phi) \notag
  \end{align}
  which shows that $W_u(\phi) \in Y_{i(u)}$.
  By \ref{r3} there there is an element $\widetilde{v}'\in
  T^\cdual$ such that for $\widetilde{v}\in T$
  \begin{equation*}
    \dup{\widetilde{v}'}{\widetilde{v}}_T
    = \int W_u(\phi)(x) \dup{\pi^\cdual(x)i(u)}{\widetilde{v}}_T \,dx.
  \end{equation*}
  Furthermore $i^\cdual(\widetilde{v}') = \phi$ in $S^\cdual$, since
  $u$ is cyclic and
  \begin{equation*}
    \dup{i^\cdual(\widetilde{v}')}{\pi(x)u}_S
    = \dup{\widetilde{v}'}{\pi(x)i(u)}_T
    = W_u(\phi)*W_{i(u)}(i(u))(x)
    = \dup{\phi}{\pi(x)u}_S
  \end{equation*}
  for each $x\in G$. This shows that $\Co_S^u Y \subseteq
  i^\cdual(\Co_T^{i(u)} Y)$.

  If on the other hand $\widetilde{v}'\in \Co_T^{i(u)} Y$, then
  \begin{equation*}
    W_u(i^\cdual(\widetilde{v}')) (x)
    =\dup{i^\cdual(\widetilde{v}')}{\pi(x)u}_S
    =\dup{\widetilde{v}'}{\pi(x)i(u)}_T
    =W_{i(u)}(\widetilde{v}')(x) \in Y
  \end{equation*}
  which shows that $i^\cdual(\widetilde{v}')\in \Co_S^u Y$.  This
  implies that $i^\cdual(\Co_T^{i(u)} Y) \subseteq \Co_S^u Y $.
  
  That the mapping $i^\cdual$ is an isometry when restricted to
  $\Co_T^{i(u)} Y$ follows directly from the calculations in
  (\ref{eq:5}).
\end{proof}

\begin{remark}
  If $(\pi,S)$ is a representation of $G$ and $u$ is a cyclic vector
  for which it is true that
  $\dup{\pi^*(x)u}{u} = \overline{\dup{u}{\pi(x)u}}$ for all
  $x\in G$ and \ref{r1} and \ref{r4} are satisfied, then
  $\dup{v}{w}$ is an inner product on $S$.  The completion
  $\mathcal{H}$ of $S$ with respect to the norm $\| v\|_\mathcal{H} =
  \sqrt{\dup{v}{v}}$ is a Hilbert space. The representation $\pi$ will then
  extend to a unitary representation $\widetilde{\pi}$ on
  $\mathcal{H}$, but we will not be able to conclude that
  $\widetilde{\pi}$ is irreducible.
\end{remark}

\begin{remark}
  In general the reproducing formula \ref{r1} does not imply 
  unitarity as shown in \cite{Zimmermann2005}. There Zimmermann
  obtains a reproducing formula from a non-unitary representation.
  It will be interesting to see if it is possible to 
  construct coorbit spaces in this setting.
\end{remark}

The following theorem is a slight generalizaton of \cite[Theorem
4.9]{Feichtinger1989a}, which in theory enables us to apply it to more general
coorbit spaces, than the ones treated in \cite{Feichtinger1989a}. 
The proof follows
that of \cite[Theorem 4.9]{Feichtinger1989a}, but we include it here for
completeness.

\begin{theorem}
  Let $Y^\cdual$ be the conjugate dual space of $Y$ and assume it is
  also a Banach space of functions.  Assume that $u\in S$ is a vector
  satisfying Assumption~\ref{assumption1} for both $Y$ and $Y^\cdual$. If the
  conjugate dual pairing on $Y^\cdual\times Y$ satisfies
  \begin{equation}\label{eq:6}
    \dup{f*W_u(u)}{g}_{Y^\cdual\times Y} = \dup{f}{g*W_u(u)}_{Y^\cdual\times Y}
  \end{equation}
  then $(\Co_S^u {Y})^\cdual = \Co_S^u (Y^\cdual)$.  If $Y$ is
  reflexive so is $\Co_S^u Y$.
\end{theorem}

If the conjugate dual pairing of $Y$ and $Y^*$ is the extension of an
integral then property (\ref{eq:6}) is true.

\begin{proof}
  Define a linear map $T:\Co_S^u (Y^\cdual) \to (\Co_S^u Y)^\cdual$ by
  \begin{equation*}
    \dup{Tw'}{\phi}_{Y^\cdual\times Y}
    = \dup{W_u(w')}{W_u(\phi)}_{(\Co_S^uY)^\cdual\times \Co_S^u Y}
  \end{equation*}
  for $w'\in \Co_S^u (Y^\cdual)$ and $\phi\in \Co_S^u Y$.  The map $T$
  is a well defined, since $W_u$ is a topological isomorphism onto its
  image.
  
  If $T(w')=0$ for some $w'\in \Co_S^u (Y^\cdual)$ then for any $f\in
  Y$ we have
  \begin{align*}
    \dup{W_u(w')}{f}_{Y^\cdual\times Y}
    &= \dup{W_u(w')*W_u(u)}{f}_{Y^\cdual\times Y}\\
    &= \dup{W_u(w')}{f*W_u(u)}_{Y^\cdual\times Y}    \\
    &= \dup{Tw'}{W_u^{-1}(f*W_u(u))}_{(\Co_S^uY)^\cdual\times \Co_S^u Y}\\
    &= 0
  \end{align*}
  since $f*W_u(u)\in W_u(\Co_S^u Y)$.  So $W_u(w')=0$ in $Y^\cdual$,
  and by the injectivity of $W_u:\Co_S^u (Y^\cdual)\to
  Y^\cdual*W_u(u)$ we conclude that $w'=0$.  This shows that $T$ is
  injective.

  Let $\widetilde{w}'\in (\Co_S^u Y)^\cdual$ and define $\widetilde{f}
  \in Y^\cdual$ by
  \begin{equation*}
    \dup{\widetilde{f}}{g}_{Y^\cdual\times Y} 
    = \dup{\widetilde{w}'}{W_u^{-1}(g*W_u(u))}_{(\Co_S^u Y)^\cdual\times \Co_S^u (Y)}
  \end{equation*}
  for all $g\in Y$.  Notice that
  \begin{equation*}
    \widetilde{f}*W_u(u) = \widetilde{f}
  \end{equation*}
  which can be seen by the calculation
  \begin{align*}
    \dup{\widetilde{f}*W_u(u)}{g}_{Y^\cdual\times Y}
    &= \dup{\widetilde{f}}{g*W_u(u)}_{Y^\cdual\times Y} \\
    &= \dup{\widetilde{w}'}{W_u^{-1}(g*W_u(u)*W_u(u))}
    _{(\Co_S^u Y)^\cdual\times \Co_S^u Y} \\
    &= \dup{\widetilde{w}'}{W_u^{-1}(g*W_u(u))}
    _{(\Co_S^u Y)^\cdual\times \Co_S^u Y} \\
    &= \dup{\widetilde{f}}{g}_{Y^\cdual\times Y}
  \end{align*}
  Thus there is a $w'\in \Co_S^u(Y^\cdual)$ such that
  \begin{equation*}
    \widetilde{f} = W_u(w')
  \end{equation*}
  Finally for all $\phi \in \Co_S^u Y$ the calculation
  \begin{align*}
    \dup{Tw'}{\phi}_{(\Co_S^u Y)^\cdual\times \Co_S^u Y}
    &= \dup{W_u(w')}{W_u(\phi)}_{Y^\cdual\times Y} \\
    &= \dup{\widetilde{f}}{W_u(\phi)}_{Y^\cdual\times Y}\\
    &= \dup{\widetilde{w}'}{W_u^{-1}(W_u(\phi)*W_u(u))}_{Y^\cdual\times Y} \\
    &= \dup{\widetilde{w}'}{W_u^{-1}(W_u(\phi))}_{Y^\cdual\times Y} \\
    &= \dup{\widetilde{w}'}{\phi}_{(\Co_S^u Y)^\cdual\times \Co_S^u Y}
  \end{align*}
  shows that $T(w') = \widetilde{w}'$ and proves that $T$ is
  surjective.
\end{proof}

We now prove that the conditions can be simplified when
dealing with unitary representations. This is worth mentioning as
some of the examples we treat later can be described in 
this manner.

\begin{theorem}
  \label{thm:coorbitsunitaryrepns}
  Let $(\pi,H)$ be a unitary representation and 
  $(S,H,S^*)$ a Gelfand triple. Let $u$ be a cyclic
  vector in $S$ such that
  $W_u(v)*W_u(u)=W_u(v)$ for all $v\in S^*$.
  Assume that for the Banach space $Y$, the mapping
  \begin{equation*}
    Y\times S\ni (F,v) \mapsto \int_G |F(x)||W_u(v)(x)|\, dx\in\mathbb{C}
  \end{equation*}
  is continuous, then $\Co_S^u Y = \{ \phi\in S^*\, |\, W_u(\phi)\in Y  \}$
  satisfies properties (\ref{prop1}-\ref{prop6}) of Theorem~\ref{mainthm}.
\end{theorem}
Note that for $Y=L^p(G)$ the requirement
is in fact a duality requirement, i.e. 
we require that $S\ni v\mapsto W_u(v)\in L^q(G)$ is continuous
for $1/p + 1/q =1$.

\begin{proof}
  The proof follows that of Theorem~\ref{mainthm}. We note that
  the requirements \ref{r1}  and \ref{r4} are used to prove the
  reproducing formula 
  $W_u(\phi)*W_u(u)=W_u(\phi)$ for all $\phi\in S^*$
  (which we instead have assumed here).

  The continuity in \ref{r2} is easily verified, since the unitarity of
  $\pi$ implies that $|W_u(u)(x^{-1})| = |W_u(u)(x)|$. Lastly
  the requirement \ref{r3} is satisfied for $F=F*W_u(u)\in Y$,
  since the continuity of
  \begin{equation*}
    S\ni v \mapsto \int_G |F(x)||W_u(v)(x)|\, dx\in\mathbb{C}
  \end{equation*}
  is assumed for all $F\in Y$.
\end{proof}

\section{Existing coorbit theories}

In this section we show that the coorbit theory of
Feichtinger and Gr\"ochenig is a special case of the coorbit
theory for dual pairs. 

\subsection{Coorbit theory by Feichtinger and Gr\"ochenig}
In the following let $(\pi,\mathcal{H})$ be a irreducible 
unitary square-integrable 
representation on a locally compact group $G$.
Then the Duflo-Moore Theorem \cite{Duflo1976} ensures that
we can choose $u\neq 0$, such that wavelet coefficients 
\begin{equation*}
  W_u(v) = \ip{v}{\pi(x)u}
\end{equation*}
satisfy a reproducing formula
\begin{equation*}
  W_u(v)*W_u(u) = W_u(v)
\end{equation*}
for all $v\in \mathcal{H}$.
Let $Y$ be a left invariant Banach function space continuously
included in $L^1_{loc}(G)$. Since convergence in $L^1_{loc}(G)$
implies convergence locally in Haar measure, the same is true
for $Y$. Define the weight 
\begin{equation*}
  w(x) = \sup_{\| F\|_Y=1} \frac{\| \ell_{x^{-1}}F \|_Y}{\| F\|_Y},
\end{equation*}
and assume that the space
\begin{equation*}
  \mathcal{H}_w^1 = \{ v\in\mathcal{H}\, |\, W_u(v)\in L^1_w   \}
\end{equation*}
contains $u$ (and thus is non-zero) and equip it with the norm
\begin{equation*}
  \| v\|_{\mathcal{H}_w^1 } = \| W_u(v)\|_{L^1_w}.
\end{equation*}
Denote the conjugate dual of 
$\mathcal{H}_w^1$ by $(\mathcal{H}_w^1)^*$, and define the
coorbit space
\begin{equation*}
  \Co_{FG}Y = \{ \phi\in (\mathcal{H}_w^1)^*\, |\, W_u(\phi)\in Y \}.
\end{equation*}
Feichtinger and Gr\"ochenig prove, among other results,
that $\Co_{FG}Y$ is a $\pi$-invariant Banach space.

Let us verify that the construction of Feichtinger and Gr\"ochenig
satisfies the assumptions of Theorem~\ref{thm:coorbitsunitaryrepns}.
The $\pi$-invariance of $\mathcal{H}_w^1$ ensures that 
$\mathcal{H}_w^1$ is dense in $\mathcal{H}$. Further $\mathcal{H}$
is weakly dense in $(\mathcal{H}_w^1)^\cdual$ 
(using the weak$*$ topology $(\mathcal{H}_w^1)^\cdual$; 
see \cite[Lemma 4.5.8(b)]{Rauhut2005}).
Thus $(\mathcal{H}_w^1,\mathcal{H},(\mathcal{H}_w^1)^\cdual)$ is a Gelfand-triple
when $(\mathcal{H}_w^1)^\cdual$ is equipped with its weak$*$ topology.
It is shown in \cite{Feichtinger1988} Corollary 4.9
that if $0\neq u\in\mathcal{H}_w^1$, then 
$u$ is cyclic in $\mathcal{H}_w^1$. 
Furhter, since $\mathcal{H}_w^1$ is contained in $\mathcal{H}$,
the reproducing formula holds for $\mathcal{H}_w^1$ and
thus \ref{r1} is satisfied. 
The mapping
\begin{equation*}
  (\mathcal{H}_w^1)^\cdual \ni 
  \phi\mapsto 
  \Big|\int \dup{\phi}{\pi(x)u}\dup{\pi(x)u}{u}\, dx \Big|
  \leq \| W_u(\phi)\|_{L^\infty_{1/w}} \| W_u(u)\|_{L^1_w}
  \in \mathbb{R}
\end{equation*}
is (strongly) continuous and thus also continuous if 
$(\mathcal{H}_w^1)^\cdual$ is equipped with its weak topology.
We have verified \ref{r4}.

The assumptions on the weight $w$ and $W_u(u)\in L^1_w$
ensure that $Y*L^1_w \subseteq Y$ and $F\mapsto |F|*|W_u(u)|$
is continuous. 
Thus the function $E=F*W_u(u)\in Y$ is reproduced
by $E=E*W_u(u)$ and it is shown in \cite{Feichtinger1989a} 
Proposition 4.3(iii) that 
\begin{equation*}
  \Big| \int F(x)W_u(u)(x^{-1})\, dx \Big|
  = |E(e)| 
  = |\dup{W_u(u)}{E}| 
  \leq C\| E\|_Y
  = C\| F*W_u(u)\|_Y
  \leq \| F\|_Y,
\end{equation*}
which proves that \ref{r2} holds.
The same proposition
tells us that $E\in L^\infty_{1/w}$ and therefore
$\pi(E)u\in (\mathcal{H}_w^1)^*$ thus proving \ref{r3}.
This shows that $\Co_{FG}Y$ is indeed a special case of
the general construction.

The proofs of the theorems by Feichtinger and Gr\"ochenig rely
on Wiener amalgam spaces, which we
briefly introduce here. For a compact neighbourhood $Q$ of
$e\in G$ let $1_Q$ be the indicator function on $Q$ and define
the control function
\begin{equation*}
  K_Q(F)(x) = \| (\ell_x 1_Q)F \|_{L^\infty}.
\end{equation*}
Then the space $W(Y)$ defined by
\begin{equation*}
  W(Y) = \{F\in Y\, |\, K_Q(F)\in Y    \}
\end{equation*}
with norm $\|F\|_{W(Y)} = \| K_Q(F)\|_Y$ 
does not depend on $Q$ (up to norm equivalence).
These spaces were 
used to verify properties \ref{r2} and \ref{r3}. 
These requirements are often easier to prove by duality 
(see Theorem~\ref{thm:coorbitsunitaryrepns})
allowing us to avoid the Wiener amalgam machinery.

\subsection{Coorbit theory for quasi-Banach spaces}
In \cite{Rauhut2007b} Rauhut introduces 
coorbits for a quasi-Banach space $Y$.
The notation in this section follows that of the
coorbit theory by Feichtinger and Gr\"ochenig. 
In order to define coorbits, Rauhut uses
Wiener amalgam spaces.
Rauhut defines the coorbit for $Y$ to be the space
\begin{equation*}
  C(Y) 
  = \{ f\in (\mathcal{H}_w^1)^*\, |\, W_u(f) \in W(Y) \}
  = \Co W(Y).
\end{equation*}
The use of $W(Y)$ ensures that the convolution by
$W_u(u)\in L^1_w$ is defined, while convolutions on
quasi-Banach spaces are generally not defined.
In \cite{Rauhut2007} is is shown that
$W(Y)$ is continuously included in $L^\infty_{1/w}(G)$ 
(which is a Banach space for which the properties in
Assumption~\ref{assumption1} can be verified with $S=\mathcal{H}_w^1$).
By Theorem~\ref{thm:subspace} it then 
follows immediately, that $C(Y)$ is a quasi-Banach space.

If $w$ is a weight for which $W_u(u)\in L^1_w$ 
it is in fact possible to define coorbit spaces for
any quasi-Banach space $Y'$ continuously included in $L^\infty_{1/w}$.
In particular the space $Y'=Y\cap L^\infty_{1/w}$ can be used.
In the case of the modulation spaces described in \cite{Galperin2004}
it turns out that $W(Y)$, $Y\cap L^\infty_{1/w}$ (and even $Y$)
give the same coorbits. If this is the case in general we do not know.

\section{Bergman spaces on the unit disc}
Let $\mathbb{D}$ be the unit disc in $\mathbb{C}$ equipped with
area measure $dz$. For $1\leq p<\infty$ and $\sigma>1$ 
the Bergman spaces are the classes of holomorphic functions 
\begin{equation*}
  A^p_\sigma({\mathbb{D}}) = 
  \Big\{ f \in\mathcal{O}({\mathbb{D}})\Big| \|f\|_{A^p_\sigma({\mathbb{D}})}^p =
  \int_{\mathbb{D}} |f(z)|^p(1-|z|^2)^{\sigma-2}\,dz < \infty   \Big\}.
\end{equation*}
In this section we give a wavelet characterization of these spaces.

\subsection{Coorbits for Discrete Series}
\label{sec:coorb-discr-seri}
Let $G\subseteq \SL$ be the connected subgroup of upper triangular matrices,
i.e.
\begin{equation*}
  G = \left\{
    \begin{pmatrix}
      a & b \\ 0 & a^{-1}
    \end{pmatrix} \,\Big|\, a>0,b\in\mathbb{R}
  \right\}
\end{equation*}
with left-invariant measure $\frac{da\,db}{a^2}$.
Through the Cayley transform this group can be regarded
as the subgroup of $\mathrm{SU}(1,1)$ consisting of matrices
\begin{equation*}
  \begin{pmatrix}
    \alpha & \beta \\ \bar\beta & \bar\alpha
  \end{pmatrix}
  =
  \frac{1}{2}
  \begin{pmatrix}
    a + a^{-1} +ib & b + i(a-a^{-1}) \\
    b - i(a-a^{-1}) & a + a^{-1} - ib 
  \end{pmatrix}.
\end{equation*}

For real numbers $s > 1$ the pairing
\begin{equation*}
  \ip{u}{v}_s 
  = \frac{s-1}{\pi} \int_\mathbb{D} u(z)\overline{v(z)}(1-|z|^2)^{s-2}  \, dz
  = \frac{s-1}{\pi} 
  \int_\mathbb{D} u(re^{i\theta})\overline{v(re^{i\theta})}
  (1-r^2)^{s-2} r\,dr \, d\theta
\end{equation*}
is an inner product on the Hilbert space
\begin{equation*}
  \mathcal{H}_s 
  = A^{2}_s(\mathbb{D})
  = \{ v\in \mathcal{O}(\mathbb{D}) | \ip{v}{v}_s<\infty \}.
\end{equation*}
The discrete series representations 
$(\pi_s,\mathcal{H}_s)$ 
are defined by
\begin{equation*}
  \pi_s
  \begin{pmatrix}
    \alpha & \beta \\ \bar\beta & \bar\alpha
  \end{pmatrix}
  v(z)
  = (-\bar\beta z + \alpha)^{-s} 
  v\Big( \frac{\bar\alpha z-\beta}{-\bar\beta z+\alpha}    \Big).
\end{equation*}
Since $G$ acts transitively on the disc $\mathbb{D}$ an argument by
Kobayashi \cite{Kobayashi1961} shows that $\pi_s$ is irreducible.
From now on we denote by $u$ the wavelet in $\mathcal{H}_s$ which
is identically $1$ on the disc
\begin{equation*}
  u(z) = 1_{\mathbb{D}}(z).
\end{equation*}
Then the wavelet coefficients $W_u^s(u)$ for $s>1$ 
can be calculated explicitly as
\begin{equation*}
  W_u^s(u)(a,b) 
  = \ip{u}{\pi_s(a,b)u} 
  = 2^{s}(a+a^{-1}-ib)^{-s}.
\end{equation*}
The following basic fact is useful to us
\begin{lemma}\label{lem:int1}
  \begin{equation*}
    \int ((a+a^{-1})^2 + b^2)^{-t} a^r
    \frac{da\,db}{a^2} < \infty
  \end{equation*}
  if and only if $2(1-t) < r < 2t$.
\end{lemma}

This shows that the representations $\pi_s$ are
square integrable for all $s>1$ and integrable for
$s>2$. That $\pi_s$ is not integrable 
for $1<s\leq 2$ turns out to
not matter for the construction of coorbit spaces for
these representations. 

Given the submultiplicative 
weight $w_r(a,b) = 2^r[(a+a^{-1})^2+b^2]^{r/2}$ for $r\geq 0$, let 
$L^p_r(G)$ denote the space 
\begin{equation*}
  L^p_r(G)
  = \left\{ f \Big| \| f \|_{L^p_r} 
    = \left( \int |f(a,b)w_r(a,b)|^p \,\frac{da\,db}{a^2}\right)^{1/p} 
    < \infty      
  \right\}.
\end{equation*}
We now construct coorbit spaces for the representations $\pi_s$ 
related to the spaces
$ L^p_r(G)$. For this we use the smooth vectors of
the representation $\pi_s$.
The following characterization of
the smooth vectors and its dual can
be found in \cite{Olafsson1988} and more generally in \cite{Chebli2004}.
\begin{lemma}
  The smooth vectors $\mathcal{H}_s^{\infty}$ for $\pi_s$
  are the power series $\sum_{k=0}^\infty a_kz^k$ for which
  there for any $m$ exists a constant $C$ such that
  \begin{equation*}
    |a_k|^2 \leq C \frac{(s+k-1)!}{(s-1)!k!}
    (1+k)^{-m}.
  \end{equation*}
  The conjugate dual $\mathcal{H}_s^{-\infty}$ of this space consists
  of formal power series $\sum_{k=0}^\infty b_kz^k$ for which
  there is an $m$ and a constant $C$ such that
  \begin{equation*}
    |b_k|^2 \leq C \frac{(s+k-1)!}{(s-1)!k!}
    (1+k)^{m}.
  \end{equation*}
\end{lemma}

By \cite{Warner1972a} p. 254 we know 
that $\mathcal{H}_s^\infty$ is irreducible if and only if
$\mathcal{H}$ is, and thus $u$ is cyclic in $\mathcal{H}_s^\infty$.
Furthermore the smooth vectors $\mathcal{H}_s^{\infty}$
are weakly dense in the dual $\mathcal{H}_s^{-\infty}$. 

\begin{theorem}\label{thm:bergmancoorbits}
  The spaces
  $\Co_{\mathcal{H}_s^\infty}^{u} L^p_r$ are non-zero
  $\pi_s$-invariant
  Banach spaces when $2-s < r+2/p < s$.
\end{theorem}

\begin{proof}
  We show first
  that the mapping 
  \begin{equation*}
  L^p_r(G)\ni  
  f\mapsto \int_G |f(a,b)||W_u^s(u)(a,b)|\,\frac{da\,db}{a^2} 
  \in \mathbb{R}
  \end{equation*}
  is continuous for $2-s < r+2/p$.
  First assume that $p>1$ and let $1/p+1/q =$, then
  \begin{align*}
    \int |f(a,b) W_u^s(u)(a,b)| \,\frac{da\,db}{a^2} 
    &=
    \int |f(a,b)|w_{r}(a,b)w_{-r}(a,b) |W_u^s(u)(a,b)| \,\frac{da\,db}{a^2}\\
    &\leq
    \int |f(a,b)w_{r}(a,b)|^p \,\frac{da\,db}{a^2} 
    \int |w_{-r}(a,b) W_u^s(u)(a,b)|^q \,\frac{da\,db}{a^2} \\
    &= C\| f\|_{L^p_r}^p 
    \int ((a+a^{-1})^2+b^2)^{-sq/2-rq/2} \,\frac{da\,db}{a^2}.
  \end{align*}
  The last integral is bounded if and only if
  $2-(s+r)q < 0$ and this can be rewritten to the condition that
  \begin{equation*}
    2-s < r+ \frac{2}{p}.
  \end{equation*}
  If $p=1$ then the integral 
  \begin{equation*}
      \int |f(a,b) W_u^s(u)(a,b)| \,\frac{da\,db}{a^2} 
      \leq \| f\|_{L^1_r} \|W_u^s(u)w_{-r} \|_{\infty}
  \end{equation*}
  is finite if $s+r\geq 0$ and in particular if $2-s < r+ \frac{2}{p}$.
  
  Next we show that
  for a given $f\in L^p_r$ the mapping 
  $$
  \mathcal{H}_s^\infty \ni v \mapsto 
  \int\int f(a,b)W_u^s(v)(a,b) \,\frac{da\,db}{a^2}
  $$
  is continuous for $2-s < r+ 2/p$. 
  Let $v$ be a smooth vector with expansion $\sum_{k=0}^\infty
  a_kz^k$.
  Since $W_u^s(z^k)= \bar\alpha^{-s} (\beta/\bar\alpha)^k$
  it can be shown that
  \begin{equation*}
    |W_u^s(v)(a,b)| 
    \leq |W_u^s(u)(a,b)| \sum_{k=0}^\infty |a_k|.
  \end{equation*}
  Therefore $|W_u^s(v)|\leq C_v |W_u^s(u)|$ where the constant 
  $C_v$ depends continuously on $v$.
  Thus we only need to require that the integral
  \begin{equation*}
    \int\int |f(a,b) W_u^s(u)(a,b)| \frac{da\,db}{a^2}
  \end{equation*}
  is finite, which we have proven above.

  Lastly, by Lemma~\ref{lem:int1} we see, that
  $u$ is in the coorbit space for $r + 2/p < s$ .
  Thus the coorbit space
  $\Co_{\mathcal{H}_s^\infty}^{u} L^p_r$ is a non-zero Banach space
  when $2-s < r+2/p < s$.
\end{proof}

In the next section we will prove that the spaces defined
in Theorem~\ref{thm:bergmancoorbits} are
in fact Bergman spaces. This was mentioned in
\cite[Section 7]{Feichtinger1988}, but not many details were given.

\subsection{Continuous Description of Bergman Spaces}
We start with a lemma
\begin{lemma}
  Assume that $2-s < r+2/p < s$. If
  $f\in A^p_{(s-r)p/2}$ then 
  $f\in \mathcal{H}_s^{-\infty}$.
\end{lemma}

\begin{proof}
  We need to estimate the coefficients $b_k$ where
  $f(z) = \sum_{k=0}^\infty b_k z^k$. For this let us first
  estimate $f^{(k)}(0)$. The condition on 
  $s$,$r$ and $p$ means in particular that
  $(s-r)p/2 -1 < (s-1)p$ and 
  we can use Theorem 1.10 in \cite{Hedenmalm2000}
  to get
  \begin{equation*}
    f(z) 
    = \frac{(s -1)}{\pi} \int_{\mathbb{D}} f(w) 
    \frac{(1-|w|^2)^{s -2}}{(1-z\bar w)^s}\, dw.
  \end{equation*}
  Differentiate under the integral sign $k$ times (which is allowed when 
  for example $|z|\leq 1/2$) 
  \begin{equation*}
    f^{(k)}(z) 
    = (s-1)s(s +1)\dots (s +k-1)
    \int_{\mathbb{D}} f(w) 
    \frac{(1-|w|^2)^{s - 2}}{(1-z\bar w)^{s+k}} \bar w^k\, dw
  \end{equation*}
  and insert $z=0$ to get
  \begin{equation*}
    f^{(k)}(0) 
    = (s-1)s(s +1)\dots (s +k-1)
    \int_{\mathbb{D}} f(w) 
    (1-|w|^2)^{s - 2}\bar w^k\, dw.
  \end{equation*}
  The absolute value of the integral can be estimated by
  \begin{align*}
    \int_{\mathbb{D}} |f(w)| (1-|w|^2)^{s-2}\, dw  
    &= \int_{\mathbb{D}} |f(w)| (1-|w|^2)^{(s-r)/2-2/p} 
    (1-|w|^2)^{(s+r)/2-2/q} \, dw  \\
%    &\leq 
%    \Big( \int_{\mathbb{D}} |f(w)|^p (1-|w|^2)^{(s-r)p/2 -2}\, dw\Big)^{1/p}
%    \Big( \int_{\mathbb{D}} (1-|w|^2)^{(s+r)q/2-2}\, dw \Big)^{1/q}
    &\leq 
    \| f\|_{A^p_{(s-r)p/2}}
    \Big( \int_{\mathbb{D}} (1-|w|^2)^{(s+r)q/2-2}\, dw \Big)^{1/q}.
  \end{align*}
  The last integral is finite when $2-s < r+2/p$, and
  therefore the coefficients $b_k$ can be estimated by
  \begin{equation*}
    |b_k| 
    = \frac{|f^{(k)}(0)|}{k!}
%    \leq C    \| f\|_{A^p_{(s-r)p/2}} 
%    \frac{(\tau-1) \tau (\tau+1)\dots (\tau +k -1)}{k!}    
    \leq C \| f\|_{A^p_{(s-r)p/2}}  
    \frac{(s +k -1)!}{(s -1)!k!}.
  \end{equation*}
  Let $\tau =\lceil s \rceil$ then we 
  can estimate
  \begin{equation*}
    \frac{(s +k -1)!}{k!}
    \leq \frac{(\tau +k -1)!}{k!}
    = \underbrace{(\tau+k-1)(\tau +k -2)\cdots (1+k)}_{\text{$\tau$ terms}}
    \leq \tau^\tau (1+k)^\tau,
  \end{equation*}
  and since $\tau$ is fixed there is a constant $C$ such that
  \begin{equation*}
    |b_k|^2
    \leq C \| f\|_{A^p_{(s-r)p/2}}  
    (1+k)^{2\tau}.
  \end{equation*}
  This shows that $f\in \mathcal{H}_s^{-\infty}$.
\end{proof}

\begin{theorem}
  The spaces $A^p_{(s-r)p/2}(\mathbb{D})$ correspond to the
  coorbits $\Co_{\mathcal{H}_s^\infty}^u L^p_r(G)$
  from Theorem~\ref{thm:bergmancoorbits}
  for $1 < (s-r)p/2 <(s-1)p +1$.
\end{theorem}

\begin{proof}
  Assume that $f\in A^p_{(s-r)p/2}(\mathbb{D})$. 
  We already know that $f\in \mathcal{H}_s^{-\infty}$, so
  we can find the wavelet coefficient of $f$
  \begin{align*}
    W_u^s(f)(a,b)
    &= \frac{s-1}{\pi}\int_{\mathbb{D}} f(z) \overline{\pi_s(a,b)u(z)} 
    (1-|z|^2)^{s-2}\,dz \\
    &= \frac{s-1}{\pi} 
    \int_{\mathbb{D}} f(z) \frac{1}{(-\beta \bar z +\bar\alpha)^s}
    (1-|z|^2)^{s-2}\,dz \\
    &=  \frac{s-1}{\pi}\frac{1}{\bar\alpha^s} 
    \int_{\mathbb{D}} f(z) \frac{1}{(1-\frac{\beta}{\bar\alpha} \bar z)^s}
    (1-|z|^2)^{s-2}\,dz \\
    &= \frac{1}{\bar\alpha^s} 
    f\Big(\frac{\beta}{\bar\alpha}   \Big).
  \end{align*}
  In the last step we
  applied Theorem 1.10 in \cite{Hedenmalm2000} 
  provided that $(s-r)p/2-1 < (s-1)p$. 
  The function $\phi:G \mapsto \mathbb{D}$ given by
  \begin{equation*}
    \phi(a,b) 
    = \frac{a^2+b^2-1}{(1+a)^2+b^2} 
    + i \frac{-2b}{(1+a)^2+b^2}
  \end{equation*}
  is a bijection, and
  $\beta/\bar\alpha$ can be rewritten as
  \begin{equation*}
    \frac{\beta}{\bar\alpha}
    = 
    \frac{2(ab)}{(1+a^2)^2+(ab)^2}
    +i \frac{(a^2)^2+(ab)^2-1}{(1+a^2)^2+(ab)^2}
    = 
    i \phi(a^2,ab).
  \end{equation*}
  Therefore 
  \begin{equation*}
    W_u^s(f)(a,b)
    = \frac{2^s}{(a+a^{-1}-ib)^s} f(i \phi(a^2,ab)).
  \end{equation*}
  Then taking $L^p_r(G)$-norm of $W_u^s(f)$ and 
  changing to an integral over the disc we get
  \begin{align*}
    \frac{1}{2^{(s+r)p}}
    \int_G  |W_u^s(f)(a,b)w_r(a,b)|^p \frac{da\,db}{a^2}
    &=
    \int_G  \frac{1}{[(a+a^{-1})^2+b^2]^{(s-r)p/2}} |f(i \phi(a^2,ab))|^p
    \frac{da\,db}{a^2} \\
    &=
    \frac{1}{2}
    \int_G  \frac{1}{[(\sqrt{a}+\sqrt{a}^{-1})^2+b^2]^{(s-r)p/2}} 
    |f(i \phi(a,\sqrt{a}b))|
    \frac{da\,db}{a\sqrt{a}} \\
    &=
    \frac{1}{2}
    \int_G  \frac{1}{[(\sqrt{a}+\sqrt{a}^{-1})^2+(\sqrt{a}^{-1}b)^2]^{(s-r)p/2}} 
    |f(i \phi(a,b))|^p
    \frac{da\,db}{a^2} \\
    &=
    \frac{1}{2}
    \int_G  \Big[ \frac{a}{(1+a)^2+b^2} \Big]^{(s-r)p/2}
    |f(i \phi(a,b))|^p
    \frac{da\,db}{a^2} \\
    &=
    2
    \int_{\mathbb{D}} (1-|z|^2)^{(s-r)p/2} 
    |f(iz)|^p
    \frac{dz}{(1-|z|^2)^2} \\
    &= \| f\|_{A^p_{(s-r)p/2}}^p.
  \end{align*}
    
  We now show that an element of the coorbit space is 
  in the Bergman space.
  Since
  any $f\in \mathcal{H}_s^{\infty}$ is in $A^2_s(\mathbb{D})$ we
  know that 
  $$W_u^s(f)(a,b) = \frac{1}{\bar\alpha^s} 
  f\Big(\frac{\beta}{\bar\alpha}   \Big)$$ 
  by Proposition 1.4 in \cite{Hedenmalm2000}.
  Because $\mathcal{H}_s^{\infty}$ is weakly dense in
  $\mathcal{H}_s^{-\infty}$, this equality also holds for 
  $f\in \mathcal{H}_s^{-\infty}$.
  Therefore the calculations above are valid, and if 
  $f\in \Co_{\mathcal{H}_s^{\infty}}^u L^p_r(G)$ then
  $f$ is also in $A^p_{(s-r)p/2}(\mathbb{D})$.
\end{proof}

\subsection{Discretization} \label{sec:discretizationbergman}
In this section we obtain sampling theorems and atomic decompositions
for the Bergman spaces by use of the wavelet transform. 
We point out that these results include
the non-integrable representations which cannot
be described by the work of Feichtinger and Gr\"ochenig.

\begin{lemma}\label{lem:bergman2}
  The mappings $f\mapsto f*W_u^s(u)$ and
  $f\mapsto f*|W_u^s(u)|$ are continuous
  $L^p_r(G)\to L^p_r(G)$ for $s > r +2/p$.
\end{lemma}

\begin{proof}
  In the following denote by $F_s$ the absolute value of
  the wavelet coefficient belongning
  to $\pi_s$, i.e. $F_s(a,b) = |\ip{u}{\pi_s(a,b)u}|$
  and notice that
  \begin{equation*}
    F_s(a,b) = w_{-s}(a,b)
  \end{equation*}
  In the calculations below we make some assumptions in order for the
  estimates to be true. At the end of the proof we collect these
  assumptions.
  
  Let $p> 1$ and
  assume that $f\in L^p_r(G)$. Let $q$ such
  that $1/p+1/q =1$ and further let $t$ be such
  that the following calculations hold (we will investigate this
  later)
  \begin{align*}
    \Big |\int\!\! \int &f(a,b) F_s((a,b)^{-1}(a_1,b_1))
    \,\frac{da\, db}{a^2}\Big|^p \\
    &\leq \Big(\int\!\! \int |f(a,b)| 
    |w_{-r-s(1/p+1/q)}((a,b)^{-1}(a_1,b_1))|^{1/p+1/q} 
    w_{r}((a,b)^{-1}(a_1,b_1))a^{t}
    \,\frac{da\, db}{a^2}\Big)^p \\
    &\leq \Big( \int\!\! \int |f(a,b)|w_{-rp-s}((a,b)^{-1}(a_1,b_1)) a^{-tp}
    \,\frac{da\, db}{a^2} \Big) \\
    &\qquad \times \Big( \int\!\! \int a^{tq}
    w_{rq-s}((a,b)^{-1}(a_1,b_1))^q \,\frac{da\, db}{a^2} \Big)^{p/q}.
  \end{align*}
  We know that $|w_r((a,b))| = |w_r((a,b)^{-1})|$ 
  so the second integral becomes
  \begin{align*}
    \int\!\! \int a^{tq}
    |w_{rq-s}((a_1,b_1)^{-1}(a,b))| \,\frac{da\, db}{a^2}
    &= 
    \int\!\! \int (aa_1)^{tq}
    |w_{rq-s}((a,b))| \,\frac{da\, db}{a^2} \\
    &= Ca_1^{tq}
  \end{align*}
  provided that $w_{rq-s}(a,b)a^{tq} \in L^1(G)$. 
  Thus we get
  \begin{align*}
    \Big |\int\!\! \int &f(a,b) F_s((a,b)^{-1}(a_1,b_1))
    \,\frac{da \,db}{a^2}\Big|^p \\
    &\leq C a_1^{tp}
    \int\!\! \int |f(a,b)|^p a^{-tp} |w_{-rp-s}((a,b)^{-1}(a_1,b_1))|
    \,\frac{da\, db}{a^2} 
  \end{align*}
  and we can estimate the norm of $f*F_s$ using Fubini's theorem
  \begin{align*}
    \| f*F_s \|_{L_r^p} 
    &\leq C \int\!\! \int\!\! \int\!\! \int |f(a,b)|^p
    a^{-tp} w_{-rp-s}((a,b)^{-1}(a_1,b_1)) \,\frac{da\, db}{a^2}
    w_{rp}(a_1,b_1)
    a_1^{tp} \frac{da_1 db_1}{a_1^2} \\
    &= C \int\!\! \int |f(a,b)|^p a^{-tp} \int\!\! \int
    |w_{-rp-s}((a,b)^{-1}(a_1,b_1))| w_{rp}(a_1,b_1) a_1^{tp}
    \,\frac{da_1 \,db_1}{a_1^2} \,\frac{da \,db}{a^2}.  
    \intertext{A change of variable and using submultiplicativity
      of the weight $w_{rp}$ then gives}
    &= C \int\!\! \int |f(a,b)|^p a^{-tp} \int\!\! \int
    w_{-rp-s}((a_1,b_1)) w_{rp}((a,b)(a_1,b_1)) (aa_1)^{tp}
    \,\frac{da_1 \,db_1}{a_1^2} \,\frac{da \,db}{a^2} \\
    &\leq C \int\!\! \int |f(a,b)|^p w_{rp}(a,b) \,\frac{da \,db}{a^2}
    \int\!\!\int w_{-s}((a_1,b_1)) a_1^{tp}
    \,\frac{da_1 \,db_1}{a_1^2} \\
    &\leq C \| f \|_{L^p_r},
  \end{align*}
  where we in the last inequality have assumed that
  $w_{-s}(a_1,b_1)a_1^{tp}\in L^1(G)$.
  
%  We now gather the assumptions made during the calculations and
%  determine when they are all true.
  To sum up the map $f\mapsto f*F_s$ is continuous
  if we are able to choose a $t$
  such that both $w_{rq-s}(a,b)a^{tq}$ and $w_{-s}(a,b)a^{tp}$ are
  in $L^1(G)$. 
  This is the case if both $2-s+rq < tq < s-rq$ and
  $2-s < tp < s$ and such a $t$ can be shown to exist
  if and only if $r+2/p < s$.
  
  For $p=1$ Fubini can be applied immediately 
  and the requirement becomes that
  $w_{r-s}(a,b)$ is integrable. This is the case if
  $2+r < s$, so the result also holds for $p=1$.
\end{proof}

The key to finding atomic decompositions will be the following
result
\begin{lemma}\label{lem:osc}
  For each $\epsilon > 0$ there is a neighbourhood 
  $U$ of the identity such
  that 
  \begin{equation*}
    \Big| \frac{W_u^s(u)((a,b)(x,y))}{W_u^s(u)(x,y)} -1 \Big| <\epsilon
  \end{equation*}
  for $(a,b)\in U$.
\end{lemma}

From this result follows easily 
\begin{corollary}\label{cor:1}
  There exist a neighbourhood $U$ of the identity
  and constants $C_1,C_2 > 0$ such that
  \begin{equation*}
    C_1 |W_u^s(u)(x,y)| \leq |W_u^s(u)((a,b)(x,y))| \leq C_2 |W_u^s(u)(x,y)|
  \end{equation*}
  for all $(x,y)\in G$ with $(a,b)\in U$.
  These constants can be chosen arbitrarily close to $1$,
  by choosing $U$ small enough.
\end{corollary}

\begin{proposition}
  Let $V\subseteq U$ be compact neighbourhoods of the identitiy.
  Assume that the points $\{ x_i\}$ are $V$-separated and $U$-dense
  and that $U$ satisfies Corollary~\ref{cor:1}. Let 
  $\{ \psi_i\}$ be a partition of unity 
  for which $\supp(\psi_i)\subseteq x_iU$.
  Define the sequence space 
  \begin{equation*}
    \ell^p_r 
    = \{ (\lambda_i)\, |\, 
    \|(\lambda_i) \|_{\ell^p_r} 
    = \Big( \sum_{i\in I} |\lambda_iw_r(x_i)|^p \Big)^{1/p}
    \}.
  \end{equation*}
  
  Then the following is true
  \begin{enumerate}
  \item The mapping 
    $\ell^p_r \ni (\lambda_i) \mapsto \sum_i \lambda_i \ell_{x_i} W_u^s(u)
    \in L^p_r(G)*W_u^s(u)$ is continuous. \label{item:1a}
  \item The mapping 
    $L^p_r(G)*W_u^s(u) \ni f \mapsto (f(x_i))_{i\in I} \in \ell^p_r(I)$
    is continuous. \label{item:2a}
  \item The mapping 
    $L^p_r(G)*W_u^s(u) \ni f\mapsto (\int_G f(x)\psi_i(x) dx)_{i\in I} \in
    \ell^p_r(I)$
    is continuous. \label{item:3a}
\end{enumerate}
\end{proposition}

As in \cite{Grochenig1991} sums are understood as 
limits of the net of 
partial sums over finite subsets with convergence in
$L^p_r(G)$.

\begin{proof}
  First note that the norms on $\ell^p_r$ and $L^p_r(G)$ are related
  in the following sense (where $C$ and $D$ are constants)
  \begin{equation*}
    \| (\lambda_i) \|_{\ell^p_r} 
    \leq C \Big\| \sum_i \lambda_i 1_{x_iV} \Big\|_{L^p_r}
    \leq D \| (\lambda_i) \|_{\ell^p_r}.
  \end{equation*}
  The convolution in $L^p_r(G)$ with $|W_u^s(u)|$ is continuous (see
  Lemma~\ref{lem:bergman2}) and we will
  denote the norm of this convolution by $D_p$.
  Further assume that we have chosen $U$ and constants
  $C_1,C_2$ satisfying Lemma~\ref{lem:osc} and Corollary~\ref{cor:1}.
  (\ref{item:1a})
  If $(\lambda_i)\in \ell^p_r$ then the function
  \begin{equation*}
    f = \sum_i |\lambda_i| 1_{x_iV}
  \end{equation*}
  is in $L^p_r(G)$ and 
  $\| f\|_{L^p_r(G)} \geq C \|(\lambda_i)\|_{\ell^p_r}$.
  Convolution with $|W_u^s(u)|$ is continuous
  so 
  \begin{equation*}
    f*|W_u^s(u)| 
    = \sum_i |\lambda_i| 1_{x_iV}*|W_u^s(u)| 
  \end{equation*}
  is in $L^p_r(G)$.
  Now let us show that the function $1_{x_iV}*|W_u^s(u)|$ is bigger
  than some constant times $\ell_{x_i}|W_u^s(u)|$.
  \begin{equation*}
    \int 1_{x_iV}(z) |W_u^s(u)(z^{-1}y)| dz 
    = \int_V |W_u^s(u)(z^{-1}x_i^{-1} y)| dz
    \geq C |W_u^s(u)(x_i^{-1}y)|.
  \end{equation*}
  This shows that
  \begin{align*}
    \Big| \sum_i \lambda_i \ell_{x_i} W_u^s(u)(y)   \Big|
    &\leq     \sum_i |\lambda_i| |W_u^s(u)(x_i^{-1}y)| \\
    &\leq C   \sum_i |\lambda_i| 1_{x_iV}*|W_u^s(u)|(y) \\
    &=  C f*|W_u^s(u)|.
  \end{align*}
  Since $f*|W_u^s(u)|\in L^p_r(G)$ the sum $\sum_i \lambda_i \ell_{x_i} W_u^s(u)$ 
  is in $L^p_r(G)$ with norm
  \begin{align*}
    \Big\| \sum_i \lambda_i \ell_{x_i} W_u^s(u)   \Big\|_{L_r^p}
    & \leq C \|f*|W_u^s(u)|\, \|_{L^p_r} \\
    & \leq C \| f\|_{L^p_r}  \\
    &\leq C \| (\lambda_i)\|_{\ell^p_r}.
  \end{align*}
  We have thus obtained the desired continuity. 
  The sum
  $\sum_i \lambda_i\ell_{x_i}W_u^s(u)$ is to be understood as a limit
  in $L^p_r(G)$ and since convolution with $W_u^s(u)$ is continuous
  we get from the reproducing formula that 
  $\sum_i \lambda_i\ell_{x_i}W_u^s(u)\in L^p_r(G)*W_u^s(u)$.

  (\ref{item:2a})
  We need to show that $\{ f(x_i)\}$ is in $\ell^p_r$, but this is the same
  as showing that $g = \sum_i |f(x_i)| 1_{x_iV}$ is in $L^p_r(G)$.
  But $f\in L^p_r(G)*W_u^s(u)$ so 
  \begin{equation*}
    \sum_i |f(x_i)|1_{x_iV}(y) 
    \leq
    \sum_i |f|*|W_u^s(u)|(x_i) 1_{x_iV}(y)
    = \int |f(z)| \sum_i |W_u^s(u)(z^{-1}x_i)| 1_{x_iV}(y) dz.
  \end{equation*}
  For each $y$ at most one $i$ adds to this sum, namely the $i$ for which
  $x_i \in yV^{-1}$. Therefore
  \begin{equation*}
      \sum_i |W_u^s(u)(z^{-1}x_i)| 1_{x_iV}(y)
      \leq \sup_{v\in V} |W_u^s(u)(z^{-1}yv^{-1})|
      \leq C_2 |W_u^s(u)(z^{-1}y)|
  \end{equation*}
  by Corollary~\ref{cor:1}. We then get
  \begin{equation*}
    \sum_i |f(x_i)|1_{x_iV}(y) 
    \leq C_2 \int |f(z)| |W_u^s(u)(z^{-1}y)| dz
    = C_2 |f|*|W_u^s(u)|(y)
  \end{equation*}
  and finally 
  \begin{equation*}
    \|f(x_i) \|_{\ell^p_r} \leq \frac{C_2D_p}{|V|} \| f\|_{L^p_r}.
  \end{equation*}
  
  (\ref{item:3a})
  We have to show that the function
  \begin{equation*}
    \sum_{i} \Big(\int f(x)\psi_i(x)dx \Big) 1_{x_iV} 
  \end{equation*}
  is in $L^p_r(G)$.
  We get that
  \begin{equation*}
    \Big| 
    \sum_{i} \Big(\int f(x)\psi_i(x)dx \Big) 1_{x_iV}(y)
    \Big|
    \leq  \int |f(x)| \sum_{i\in I} \psi_i(x) 1_{x_iV}(y) dx
  \end{equation*}
  and since 
  \begin{equation*}
    \sum_{i} \psi_i(x)1_{x_iV}(y) 
    \leq \sum_{i} 1_{x_iU}(x)1_{x_iV}(y)
    \leq 1_{U^{-1}V}(x^{-1}y)
  \end{equation*}
  we obtain
  \begin{equation*}
    \Big| 
    \sum_{i} \Big(\int f(x)\psi_i(x)dx \Big) 1_{x_iV}(y)
    \Big|
    \leq
    \int |f(x)|1_{U^{-1}V}(x^{-1}y) dx
    = |f|*1_{U^{-1}V}(y).
  \end{equation*}
  The last function 
  is in $L^p_r(G)$ with norm continuously dependent on $f$, i.e.
  \begin{equation*}
    \Big\| \sum_{i} \Big(\int f(x)\psi_i(x)dx \Big) 1_{x_iV}
    \Big\|
    \leq C \| f\|_{L^p_r}
  \end{equation*}
  for some $C>0$.
\end{proof}

\begin{proposition}
  We can choose a compact neighbourhood $U$, 
  $U$-dense points $\{ x_i \}$ and a partition
  $\psi_i$ of unity with $\supp(\psi_i)\subseteq x_iU$
  such that the operators 
  $T_1,T_2,T_3: L^p_r(G)*W_u^s(u) \to L^p_r(G)*W_u^s(u)$ 
  defined below are invertible with
  continuous inverses
  \begin{enumerate}
  \item $T_1f = \sum_i f(x_i)\psi_i*W_u^s(u)$
    \label{item:4a}
  \item $T_2f = \sum_i c_if(x_i)\ell_{x_i}W_u^s(u)$
    with $c_i = \int \psi_i$)
    \label{item:5a}
  \item $T_3f = \sum_i \Big( \int f(x)\psi_i(x)\, dx\Big)\, \ell_{x_i}W_u^s(u)$
    \label{item:6a}
  \end{enumerate}
\end{proposition}

\begin{proof}
  For each neighbourhood of the identity  
  $U$ we can pick $U$-dense points $\{ x_i\}$
  such that $\{ x_i \}$ are $V$-separated for some
  compact neighbourhood of the identity 
  $V$ satisfying $V^2\subseteq U$ (see \cite[Thm 4.2.2]{Rauhut2005}).
  Thus we can pick $U$ in order to satisfy the inequality
  in Lemma~\ref{lem:osc} for any $\epsilon$.

  Denote by $D_p$ the $L^p_r$ operator norm of convolution by
  $W_u^s(u)$.

  (\ref{item:4a})
  Let $f\in L^p_r(G)*W_u^s(u)$ and investigate the difference
  \begin{equation*}
    f(x) - \sum_i f(x_i)\psi_i(x) = \sum_i (f(x)-f(x_i))\psi_i
  \end{equation*}
  For $x\in \supp(\psi_i) \subseteq x_i U$ we get
  \begin{align*}
    |f(x)-f(x_i)|
    &\leq \int |f(z)| |W_u^s(u)(z^{-1}x) - W_u^s(u)(z^{-1}x_i)| dz \\
    &= \int |f(z)|  
    \Big| \frac{W_u^s(u)(z^{-1}x_i)}{ W_u^s(u)(z^{-1}x)} -1 \Big|
    |W_u^s(u)(z^{-1}x)| dz
    &\leq \epsilon \int |f(z)| |W_u^s(u)(z^{-1}x)| dz \\
    &= \epsilon |f|*|W_u^s(u)|(x).
  \end{align*}
  This means that
  \begin{equation*}
    |f(x) - \sum_i f(x_i)\psi_i(x)|
    \leq 
    \epsilon \sum_i |f|*|W_u^s(u)|(x) \psi_i(x)
    = \epsilon |f|*|W_u^s(u)|(x).
  \end{equation*}
  This function is in $L^p_r(G)$ and
  \begin{equation*}
    \Big \| f - \sum_i f(x_i)\psi_i \Big\|_{L^p_r}
    \leq \epsilon D_p \|f \|_{L^p_r}.
  \end{equation*}
  Convoluting this expression by $W_u^s(u)$ we get
  \begin{equation*}
    \Big \| f - \sum_i f(x_i)\psi_i*W_u^s(u) \Big\|_{L^p}
    \leq \epsilon D_p^2 \|f \|_{L^p_r}.
  \end{equation*}
  Letting $U$ be such that $\epsilon < D_p^{-2}$, we obtain 
  an operator $T_1$ such that 
  $\| I-T_1\| < 1$ as an operator on $L^p_r(G)*W_u^s(u)$.
  Therefore $T_1$ is invertible.

  (\ref{item:5a}) and (\ref{item:6a})
  $T_2$ and $T_3$ can be shown to be invertible by
  calculating their difference from
  the operator $T_1$. 
\end{proof}

\begin{remark}
  Note that we have avoided the use of integrability in the 
  calculations above. This allows us to treat the cases $1<s\leq 2$ which
  could not be treated in \cite{Feichtinger1988}. 
\end{remark}

\section{A Wavelet Characterization of 
  Besov spaces on the forward light cone}

\renewcommand{\det}{\mathrm{Det}}

The classical wavelet transform is related to the
group $\mathbb{R}_+\ltimes\mathbb{R}$ and the 
representation 
$\pi(a,b)f(x) = \frac{1}{\sqrt{a}}f(a^{-1}(x-b))$.
In the present section we replace $\mathbb{R}_+$ with the 
group $\mathbb{R}_+\mathrm{SO}_0(n-1,1)$ acting transitively on the 
forward light cone in $\mathbb{R}^n$. This leads to the
construction of wavelets and coorbits for the group
$\mathrm{SO}_0(n-1,1)\ltimes \mathbb{R}^n$.
We show that the constructed coorbits correspond to
Besov spaces for the forward light cone introduced
in \cite{Bekolle2004}. The representations involved are integrable
and thus the theory of Feichtinger and Gr\"ochenig is sufficient. 
However we find the construction interesting enough to be included here.

\subsection{Wavelets and coorbits on the forward light cone}
We review the wavelet transform already studied in \cite{Bernier1996} 
and \cite{Fabec2003}, and introduce coorbit spaces related to
the forward light cone. 

Let $B(x,y)$ be the bilinear form on $\mathbb{R}^n$ given by
$$
B(x,y) = x_ny_n-x_{n-1}y_{n-1}-\dots-x_1y_1
$$
and let
$\mathrm{SO}_0(n-1,1)$ be the closed connected subgroup of 
$\mathrm{GL}(n,\mathbb{R})$
which leaves $B$ invariant.
The group $\mathrm{SO}_0(n-1,1)$ has the
Iwasawa decomposition $ANK$
\begin{align*}
  A &= \left\{ a_t =
  \begin{pmatrix}
    \cosh t & 0 & \sinh t \\ 0 & I_{n-2} & 0 \\ \sinh t & 0 & \cosh t
  \end{pmatrix}
  \Big| t\in\mathbb{R} \right\}\\
  N &= \left\{ n_c =
  \begin{pmatrix}
    1-|c|^2/2 & -c^{T} & |c|^2/2 \\ 
    c & I_{n-2} & -c \\ -|c|^2/2 & -c^T & 1+|c|^2/2
  \end{pmatrix}
  \Big| c\in\mathbb{R}^{n-2} \right\} \\
  K &= \left\{ k_\sigma =
  \begin{pmatrix}
    \sigma & 0 \\ 0 & 1
  \end{pmatrix}
  \Big| \sigma \in \mathrm{SO}(n-1) \right\}
\end{align*}
where $c^T$ means the transpose of $c$.

The forward light cone is the subset $\Lambda$ of $\mathbb{R}^n$
satisfying 
\begin{equation*}
  \Lambda= \{(x_1,\dots,x_n)\, |\, B(x,x) > 0,
  x_n> 0 \}
\end{equation*}
with determinant given by
\begin{equation*}
  \det(x) = \sqrt{B(x,x)}.
\end{equation*}

An element 
$\gamma a_t n_c k_\sigma\in \mathbb{R}_+\mathrm{SO}_0(n-1,1)$
acts from the left on $x\in\Lambda$ by matrix multiplication, i.e.
$\gamma a_t n_c k_\sigma x$. This action is transitive on $\Lambda$
and the measure $\det(x)^{-n}\, dx$, 
where $dx$ is the Lebesgue measure
on $\mathbb{R}^n$, is $\mathbb{R}_+\mathrm{SO}_0(n-1,1)$-invariant.
The left-regular representation of $\mathbb{R}_+\mathrm{SO}_0(n-1,1)$ 
on $L^2(\Lambda)$ is
\begin{equation*}
   \ell(\gamma a_t n_c k_\sigma) f(x)
   = f((\gamma a_t n_c k_\sigma)^{-1}x).
\end{equation*}
The subgroup $K$ leaves the base point $e=(0,\dots,0,1)^T$ invariant
and therefore the group $H=\mathbb{R}_+AN$ acts simply transitively
on the forward light cone, i.e.
every $x\in \Lambda$ can be written
$x=\gamma a_tn_c e$. 
In particular if
\begin{equation*}
  \gamma a_tn_c e =
  \gamma \begin{pmatrix}
    \sinh t + e^t|c|^2/2 \\
    -c \\
    \cosh t +e^t|c|^2/2
  \end{pmatrix}
  =
  \begin{pmatrix}
    x_1 \\
    \vdots\\
    x_n
  \end{pmatrix}
\end{equation*}
then $\gamma$,$t$ and $c$ are determined uniquely by
\begin{equation*}
  \gamma = \det(x),\,
  c = -\gamma^{-1} ( x_2,\dots, x_{n-1})^T,\,
  t = -\ln (\gamma^{-1}(x_n-x_1)),
\end{equation*}

The left invariant measure on $H$ is given by
\begin{equation*}
  \int_H f(h) dh 
  = \int_H f(\gamma a_tn_c) \,\frac{d\gamma\,dc\,dt}{\gamma}
\end{equation*}
where $dt$,$dc$ and $d\gamma$ are the Lebesgue measures
on $\mathbb{R}$,$\mathbb{R}^{n-2}$ and $\mathbb{R}_+$ respectively.
We can pass from an integral over the cone
to an integral over the group by
\begin{equation*}
  \int_{\Lambda} f(x) \frac{dx}{\det(x)^{n}} 
  =  \int_H f(\gamma a_tn_c e) \,\frac{d\gamma\,dc\,dt}{\gamma}.
\end{equation*}
An integral over the light cone with respect to Lebesgue 
mesure can therefore be writen as an integral over the group
in the following way
\begin{equation*}
  \int_{\Lambda} f(x) \,dx
  = \int_{H} f(\gamma a_tn_c e) \gamma^{n-1}\,d\gamma\, dc\, dt.
\end{equation*}
The right Haar measure on $H$ is
given by
\begin{equation*}
  f\mapsto \int_{\mathbb{R}_+\times\mathbb{R}\times \mathbb{R}^{n-2}} 
  f(\gamma a_t n_c) e^{(n-2)t} \frac{d\gamma\,dt\, dc}{\gamma}.
\end{equation*}
The modular function on $H$ is then 
$\Delta(\lambda a_t n_c) = e^{(n-2)t}$ satisfying
\begin{equation*}
  \int_H f(\gamma a_t n_c \lambda a_{t_1} n_{c_1}) \frac{d\gamma\,dt\, dc}{\gamma}
  = \Delta(\lambda a_{t_1} n_{c_1})  
  \int_H f(\gamma a_t n_c) \frac{d\gamma\,dt\, dc}{\gamma}
\end{equation*}
and
\begin{equation*}
  \int_H f((\gamma a_t n_c)^{-1} ) \frac{d\gamma\,dt\, dc}{\gamma}
  = \int_H f(\gamma a_t n_c ) \Delta(\gamma a_t n_c)^{-1}
  \frac{d\gamma\,dt\, dc}{\gamma}.
\end{equation*}

Introduce the Fourier transform related
to the bilinear form $B$ by letting
\begin{equation*}
  \widetilde{f}(w) = \widetilde{\mathcal{F}}(f)(w)
  = \frac{1}{\sqrt{2\pi}^n}\int_{\mathbb{R}^n} f(x)e^{-iB(x,w)}\, dx
\end{equation*}
for $f\in L^1(\mathbb{R}^n)$.
We know that $\widetilde{\mathcal{F}}$ extends to a 
unitary operator on $L^2(\mathbb{R}^n)$ 
and is a topological isomorphism from $\mathcal{S}(\mathbb{R}^n)$
onto $\mathcal{S}(\mathbb{R}^n)$.
It acts on convolutions like the usual Fourier transform
\begin{equation*}
  \widetilde{f*g}(w) = \sqrt{2\pi}^n\widetilde{f}(w)\widetilde{g}(w).
\end{equation*}
Denote by $\mathcal{S}(\mathbb{R}^n)$ the space of rapidly decreasing
smooth functions with topology induced by the semi-norms
\begin{equation*}
  \| f\|_{k,l} 
  = \sup_{|\alpha|\leq k} \sup_{x\in\mathbb{R}^n} |D^{\alpha} f (x)|(1+|x|^2)^l.
\end{equation*}
Here $\alpha$ is a multi-index and $k,l\geq 0$ are integers.
Since 
$\widetilde{\mathcal{F}}:\mathcal{S}(\mathbb{R}^n)\to\mathcal{S}(\mathbb{R}^n)$
is a topological isomorphism, we can extend the Fourier transform to
tempered distributions in the usual way. In this text we work with 
the conjugate dual $\mathcal{S}^*(\mathbb{R}^n)$
of $\mathcal{S}(\mathbb{R}^n)$ (in order for it to
resemble an inner product) and thus we define the Fourier
transform $\widetilde{\phi}$ for $\phi\in \mathcal{S}^*(\mathbb{R}^n)$ by
\begin{equation*}
  \dup{\widetilde{\phi}}{\widetilde{f}} = \dup{\phi}{f}
\end{equation*}

The group $G=H\ltimes \mathbb{R}^n$ 
has a natural representation on 
\begin{equation*}
  L^2_\Lambda = \{ f\in L^2(\mathbb{R}^n) |
  \supp(\widetilde{f})\subseteq \Lambda \}
\end{equation*}
given by
\begin{equation*}
  \pi(\gamma a_tn_c,b) f(x)
  = \frac{1}{\gamma^{n/2}}f((\gamma a_tn_c)^{-1}(x-b)).
\end{equation*}
This generalizes the quasi-regular representation
of the group $\mathbb{R}_+\ltimes\mathbb{R}$ from the classical
wavelet transform.
In the Fourier domain this representation becomes
\begin{equation*}
  \widetilde{\pi}(\gamma a_tn_c,b) \widetilde{f}(w)
  = {\gamma^{n/2}}
  \widetilde{f}(\gamma (a_tn_c)^{-1}w))e^{-iB(b,w)}
\end{equation*}
and we recognize that it
arises from the left action of $H$ on 
the cone $\Lambda$, and that $\widetilde{\mathcal{F}}$ is 
an intertwining operator.
The group $G$ has left invariant measure given by
\begin{equation*}
  \int_G f(g) dg = \int f(\gamma a_tn_c,b) 
  \,\frac{d\gamma\,dc\,dt\,db}{\gamma^{n+1}}.
\end{equation*}

The following result has a generalization to symmetric cones
(see for example {\cite{Fabec2003} and \cite{Bernier1996}})
and ensures that wavelets
for this representation exist.
\begin{theorem}
  The representation $(\pi,L^2_\Lambda)$ is square-integrable.
\end{theorem}

We introduce the space $\mathcal{S}_\Lambda$
of rapidly decreasing functions whose Fourier transform is supported
on the closure of the cone, i.e.
\begin{equation*}
  \mathcal{S}_\Lambda 
  = \{f\in \mathcal{S}(\mathbb{R}^n)
  \, |\, \supp(\widetilde{f})\subseteq \overline{\Lambda}    \}.
\end{equation*}
This space will be equipped with the subspace topology it inherits from
$\mathcal{S}(\mathbb{R}^n)$. The representation $\pi$ can be
restricted to $\mathcal{S}_\Lambda$ and we denote the resulting representation
by $(\pi,\mathcal{S}_\Lambda)$ or simply $\pi$.
\begin{lemma}
  Let $u\in \mathcal{S}_\Lambda$ be compactly supported
  such that $0 \leq \widetilde{u}\leq 1$ 
  and also $1/2 < \widetilde{u} \leq 1$ on a neighborhood $U$ of $e$.
  Then $u$ is cyclic in 
  $(\pi,\mathcal{S}_\Lambda)$.
  Further let
  $$
  C_u = \int_\Lambda |\widetilde{u}(w)|^2 
  \det(w)^{2(1-n)} (w_n-w_1)^{n-2}\, dw.
  $$
  Then the reproducing formula
    $$W_u(v)*W_u(u)= C_u W_u(v)$$
  holds for $v\in \mathcal{S}_\Lambda$.
\end{lemma}

\begin{proof}
  The Fourier transform $\widetilde{\mathcal{F}}$ has
  the same properties as the usual Fourier transform.
  The calculations below are immediate adaptations of results
  found in for example \cite[Chapter 6 and 7]{Rudin1991}.
  
  Let $L$ be in the conjugate dual of
  $\mathcal{S}_\Lambda$ and assume that
  $\dup{L}{\pi(\gamma a_t n_c,b)u} = 0$ for 
  all $(\gamma a_t n_c,b)\in G$. Then
  the Fourier transform can be used to obtain
  $$\dup{\widetilde{L}}{\widetilde{\pi}
    (\gamma a_t n_c,b)\widetilde{u}} = 0.$$
  Let $e_b(w) = e^{-iB(b,w)}$. The equation above can be rewritten as
  $$ 
  0 = \dup{\widetilde{L}}{e_b \widetilde{\pi}
      (\gamma a_t n_c,0)\widetilde{u}}
  = \dup{(\overline{\widetilde{\pi}(\gamma a_t n_c,0)
      \widetilde{u}})\widetilde{L}}{e_b }
  $$
  which shows that the compactly supported functional
  $(\overline{\widetilde{\pi}(\gamma a_t n_c,0)
    \widetilde{u}})\widetilde{L}$ is equal to $0$
  (see \cite[Theorem 7.23]{Rudin1991}).
  This means that for all $v \in \mathcal{S}_\Lambda$ for which
  $\widetilde{v}$ has compact support $C\subseteq \Lambda$ 
  we have the equalities
  $$
  \dup{\widetilde{L}}{\widetilde{v}\widetilde{\pi}(\gamma a_t n_c,0)
    \widetilde{u}}
  =
  \dup{\overline{\widetilde{\pi}(\gamma a_t n_c,0)
      \widetilde{u}}\widetilde{L}}{\widetilde{v}}
  =0.
  $$
  We will now show that $\dup{v}{L}$ is also $0$.
  Since $C$ is compact we can cover $C$ by a finite number of translates
  of $U$
  $$C\subseteq \bigcup_{i=1}^m(\gamma a_t n_c)_i U$$
  Define the function 
  $$\Psi = \sum_{i=1}^m
  \pi((\gamma a_t n_c)_i,0)\widetilde{u}$$
  which has support containing $C$ (here we use that $\widetilde{u}$
  is bounded away from $0$ on $U$). Then 
  $\widetilde{v}/\Psi$
  is in $C_c^\infty$ and we see that
  $$\dup{\widetilde{L}}{\widetilde{v}}
  =\dup{\Psi \widetilde{L}}{v/\Psi}
  =\sum_{i=1}^n \dup{
    \overline{\widetilde{\pi}((\gamma a_t n_c)_i,0)\widetilde{u} }
    \widetilde{L}}{v/\Psi} 
  =0. 
  $$
  Lastly, any function in $\mathcal{S}_\Lambda$ can be 
  approximated by a function whose Fourier transform is compact, and
  therefore $L=0$ in the conjugate dual of $\mathcal{S}_\Lambda$.
\end{proof}

We will need the following lemma, which corresponds to Lemma 3.11 in 
\cite{Bekolle2004}. 

\begin{lemma}
  \label{lem:rapidoncone}
  If $f\in \mathcal{S}_\Lambda$ and $k,l$ are non-negative 
  integers then there is a constant $C_{k}$ such that
  \begin{equation*}
    |\widetilde{f}(w)| 
    \leq C_{k}\| \widetilde{f}\|_{k,l} \frac{\det(w)^{k}}{(1+|w|^2)^l}.
  \end{equation*}
\end{lemma}

We will further need an estimate of the wavelet coefficients
of Schwartz functions. The estimate
actually shows that the wavelet coefficients are integrable.

\begin{lemma}
  The mapping 
  $$
  \mathcal{S}_\Lambda \ni v\mapsto 
  \int_G |W_u(v)(\gamma a_t n_c,b)| 
  \gamma^r \frac{d\gamma\,dt\, dc\,db}{\gamma^{n+1}}
  \in \mathbb{R}^+
  $$
  is continuous for all $r\in \mathbb{R}$.
\end{lemma}

\begin{proof}
  First note that the wavelet coefficients can be rewritten as
  \begin{align*}
    W_u(f)(\gamma a_t n_c,b)
    &= \ip{f}{\pi(\gamma a_t n_c,b)u} \\
    &= \gamma^{n/2} \int_\Lambda \widetilde{f}(w) 
    \widetilde{u}(\gamma n_{-c}a_{-t}w)e^{-iB(w,b)} \,dw \\
    &= -b_i^{-2} \gamma^{n/2} \int_\Lambda 
    \frac{\partial^2}{\partial w_i^2}[\widetilde{f}(w) 
    \widetilde{u}(\gamma n_{-c}a_{-t}w)] e^{-iB(w,b)} \,dw \\
  \end{align*}
  where we have used integration by parts twice.
  Therefore, if 
  $L = -\sum_{k=1}^n \frac{\partial^2 }{\partial w_k^2}$ is
  the Laplacian 
  we obtain
  \begin{equation*}
    (1+|b|^2) W_u(f)(\gamma a_t n_c,b) 
    = \gamma^{n/2} \int_\Lambda 
    ( 1+L) [\widetilde{f}(w)\widetilde{u}(\gamma n_{-c}a_{-t}w)]
    e^{-iB(w,b)} \,dw \\
  \end{equation*}
  If we repeat the argument we are able to obtain
  \begin{equation*}
     |W_u(f)(\gamma a_t n_c,b)|
    \leq (1+|b|^2)^{-N} \gamma^{n/2} \int_\Lambda
    |(1+L)^N [\widetilde{f}( w)\widetilde{u}
    (\gamma n_{-c} a_{-t} w)]| \,dw
  \end{equation*}
  for any $N$, thus proving that the wavelet coefficients are 
  indeed integrable in $b$.
  We have
  \begin{equation*}
    (1+L)^N [\widetilde{f}( w)\widetilde{u}
    (\gamma n_{-c} a_{-t} w)]
    = \sum_{|\alpha+\beta| \leq 2N}
    p_{\beta}(\gamma n_{-c}a_{-t})
    \partial^\alpha \widetilde{f}(w)
    \partial^\beta \widetilde{u} (\gamma n_{-c}a_{-t}w)
  \end{equation*}
  where $\alpha,\beta$ are multi-indices and 
  $p_{\beta}(\gamma n_{-c}a_{-t})$ are polynomials 
  in the entries of the matrix $\gamma n_{-c}a_{-t}$ (see the form
  of this matrix in \eqref{eq:1}).
  We thus have to show the integrability in $\gamma$,$t$ and $v$ 
  of expressions of the form
  \begin{equation*}
    |p_\beta(\gamma n_{-c}a_{-t})|\gamma^{n/2}
    \int_\Lambda |
    \partial^\alpha \widetilde{f}(w)
    \partial^\beta \widetilde{u} (\gamma n_{-c}a_{-t}w) |
     \,dw
  \end{equation*}
  A change of variable and use of the fact that $C=\supp({\widetilde{u}})$
  is compact, reduces this to show that
  \begin{equation*}
    \gamma^{-n/2}
    |p_\beta(\gamma n_{-c}a_{-t})| 
    \|\partial^\beta \widetilde{u} \|_{\infty}
    \int_C |\partial^\alpha \widetilde{f}(\gamma^{-1} a_t n_c w)|
    \,dw
  \end{equation*}
  is integrable.
  By Lemma~\ref{lem:rapidoncone} we can estimate any such expression
  by 
  \begin{equation}
    \label{eq:2}
    C_k\gamma^{-n/2}
    |p_\beta(\gamma n_{-c}a_{-t})| 
    \|\partial^\beta \widetilde{u} \|_{\infty}    
    \| \partial^{\alpha}\widetilde{f}\|_{k,l}
    \int_C
    \frac{\det(\gamma^{-1}w)^{k}}{(1+|\gamma^{-1}a_t n_cw|^2)^{l}}
    \gamma^{n/2}\,dw
  \end{equation}
  for arbitrary $k,l$. 
  Since the set $C$ is compact, $w$ is bounded away from $0$ and we
  see that 
  $$
  C_1(1+|\gamma^{-1}a_t n_ce|^2)
  \leq 1+|\gamma^{-1}a_t n_c w|^2
  \leq C_2(1+|\gamma^{-1}a_t n_ce|^2),
  $$
  and we can estimate \eqref{eq:2} by
  \begin{equation*}
    C_{k}\|\partial^\beta \widetilde{u} \|_{\infty}    
    \|\partial^{\alpha}\widetilde{v}\|_{k,l}
    \frac{|p_\beta(\gamma n_{-c}a_{-t})| \gamma^{-k-n/2}}{
      (1+|\gamma^{-1}a_t n_ce|^2)^{l}}.
  \end{equation*}
  All that is left now is to show that %how to choose
  %$k$ and $l$ such that 
  \begin{equation}
    \label{eq:9}
    \int_H \frac{|p_\beta(\gamma n_{-c}a_{-t})| \gamma^{-k-n/2}}{
      (1+|\gamma^{-1}a_t n_ce|^2)^{l}} 
    \gamma^r\frac{d\gamma\,dt\,dc}{\gamma^{n+1}} < \infty.
  \end{equation}
  We split this integral into two cases.
  
  \textbf{Case 1:}
  $0< \gamma \leq 1$ \\ %and $r > k+3n/2$. \\
  The expression \eqref{eq:9} can be estimated by
  \begin{align*}
    \int_{\mathbb{R}^{n-2}} \int_\mathbb{R} &\int_0^1 
    \frac{|p_\beta(n_{-c}a_{-t})| \gamma^{r-k-3n/2-1}}{
      |\gamma^{-1}a_t n_ce|^{2l}} \,d\gamma\,dt\,dc \\
    &=
     \int_{\mathbb{R}^{n-2}} \int_\mathbb{R} \int_0^1 
     \frac{|p_\beta(n_{-c}a_{-t})| \gamma^{2l+r-k-3n/2-1}}{
      |a_t n_ce|^{2l}} \,d\gamma\,dt\,dc.
  \end{align*}
  The integral over $\gamma$ is finite for $0<\gamma \leq 1$ if
  $l$ is chosen large enough (and $k=0$).
  Now
  \begin{align}
    n_{-c}a_{-t}
    &= 
    \begin{pmatrix}
      \cosh t -e^t |c|^2 & c^T & -\sinh t+  e^t |c|^2 \\
      -e^t c & I & e^tc \\
      -\sinh t -e^t |c|^2 & c^T & \cosh t + e^t|c|^2
    \end{pmatrix} \label{eq:1}
    \intertext{and}
    a_tn_{c}e
    &= 
    \begin{pmatrix}
      \sinh t+ e^t |c|^2 \\
      -c \\
      \cosh t + e^t|c|^2
    \end{pmatrix},\notag
  \end{align}
  so we see that $|p_\beta(n_{-c}a_{-t})|$ 
  will be dominated by $|a_tn_ce|^{2l} \geq 1$ if $l>\mathrm{deg}(p_\beta)$.
  Thus choosing $l$ large enough, the integral in
  \eqref{eq:9} will be finite.
  
\textbf{Case 2:}
  $\gamma \geq 1$. \\
  In this case 
  \eqref{eq:9} can be estimated by
  \begin{equation*}
     \int_{\mathbb{R}^{n-2}} \int_\mathbb{R}
     \frac{|p_\beta(n_{-c}a_{-t})|}{|a_t n_ce|^{2l}} 
    \,dt\,dc
    \int_1^\infty  
    \gamma^{r+2l+\mathrm{deg}(p_\beta)-k-3n/2-1}
    \,d\gamma.
  \end{equation*}
  The first integral is finite if $l$ is large enough, and the
  second integral is finite when $k$ is chosen large enough (depending on $l$).

To sum up we have obtained the following estimate
$$
\int_G |W_u(v)(\gamma a_t n_c,b)| 
\gamma^r \frac{d\gamma\,dt\, dc\,db}{\gamma^{n+1}}
\leq C \sum_{|\alpha+\beta|\leq 2N} 
\|\partial^\beta \widetilde{u} \|_{\infty} 
\|\partial^{\alpha}\widetilde{v} \|_{k,l},
$$
which shows the continuous dependence on $v$.
\end{proof}

Denote by
$L^{p,q}_s(G)$ the space of measurable functions $f$ on the group
for which
\begin{equation*}
  \| f\|_{L^{p,q}_s} = 
  \Big( \int_{H} \Big( \int_{\mathbb{R}^n} 
  |f(\gamma a_t n_c,b)|^p \,db \Big)^{q/p}
  \gamma^s\frac{\,d\gamma\,dt\,dc}{\gamma^{n+1}}\Big)^{1/q} < \infty
\end{equation*}
then the integrability of $W_u(v)$ shows that
\begin{lemma}\label{lemma:convcontcone}
  For $u,v\in\mathcal{S}_\Lambda$ it holds that
  $L^{p,q}_s*W_u(v) \subseteq L^{p,q}_s$ and
  $$L^{p,q}_s\ni F \mapsto F*W_{u}(v) \in L^{p,q}_s$$
  is continuous.
\end{lemma}
Further the integrability also shows that
for $1/p+1/p' =1$ and $1/q +1/q'=1$ the 
wavelet coefficient $W_u(v)$ 
is in $L^{p',q'}_{1/s}$ and therefore
\begin{lemma}
  The mapping 
  $$\mathcal{S}_\Lambda \ni v\mapsto \int_G |F(\gamma a_t n_c,b)
  W_u(v)(\gamma a_t n_c,b)| \,\frac{d\gamma\,dt\,dc}{\gamma^{n+1}}
  \in\mathbb{C}$$ is
  continuous for all $F\in L^{p,q}_s$.
\end{lemma}

This verifies the assumptions for construction of coorbit spaces
for the spaces $L^{p,q}_s$ and therefore we can define
\begin{equation*}
  \Co_{\mathcal{S}_\Lambda}^u L^{p,q}_s 
  = \{ \Phi \in \mathcal{S}_\Lambda' \,|\, W_u(\Phi) \in L^{p,q}_s \}
\end{equation*}
Furthermore, Lemma~\ref{lemma:convcontcone} shows that
this space is independent on the wavelet $u$.

\begin{remark}[Discretization]
  The representation used for this construction is integrable (as we
  have shown) and therefore the discretization procedure by
  Feichtinger and Gr\"ochenig can be used directly. 
\end{remark}

\subsection{Besov spaces as coorbits}

In this section we
introduce Littlewood-Paley decompositions and a family of
Besov spaces related to these decompositions.
The construction has been carried out for
all symmetric cones in \cite{Bekolle2004} and
we refer to this article for proofs.
The last result of this paper is a wavelet description of the 
Besov spaces. In particular we
show that the Besov spaces are the coorbit spaces defined in 
the previous section.

The group $\mathbb{R}_+A$ is an abelian group with
exponential function 
$\exp : \mathbb{R}\times\mathbb{R} \to \mathbb{R}_+A$
given by $\exp(t,s) = e^t a_s$ (here $e^t$ denotes the usual
exponential function on $\mathbb{R}$).
Let $V_r = \{(s,t)\in \mathbb{R}\times\mathbb{R} | s^2+t^2 < r \}$
and define the $K$-invariant ball $B_r(e) = K\exp(V_r)e \subseteq \Lambda$. 
Then for $w=he \in \Lambda$ with $h\in H$ we define the ball
of radius $r$ centered at $w$ to be
\begin{equation*}
  B_r(w) = h B_r(e)
\end{equation*}

The following covering lemma for the cone can be extracted from
Lemma~2.6 in \cite{Bekolle2004} and is illustrated in
Figure~\ref{fig:conecover}.
\begin{lemma}[Whitney cover with lattice points $w_j$]
  \label{lem:covercone}
  Given $\delta>0$, there exists a sequence $\{ w_j\}\subseteq \Lambda$ 
  such that $B_{\delta/2}(w_j)$ are disjoint and $B_\delta(w_j)$ cover
  $\Lambda$ with the property that there is an $N$ such that
  any $w\in \Lambda$ belongs to at most $N$ of the balls $B_\delta(w_j)$ 
  (finite intersection property).
\end{lemma}

\begin{figure}[h]
  \centering
  \subfigure[The ball $B_r(e)$]{
    % \tiny
    % \psfrag{x1}[c]{$x_1$}
    % \psfrag{xn}[l]{$x_n$}
    % \psfrag{e}[l]{$e$}
    % \psfrag{B}[l]{$B_r(e)$}
    \includegraphics[width=0.45\columnwidth]{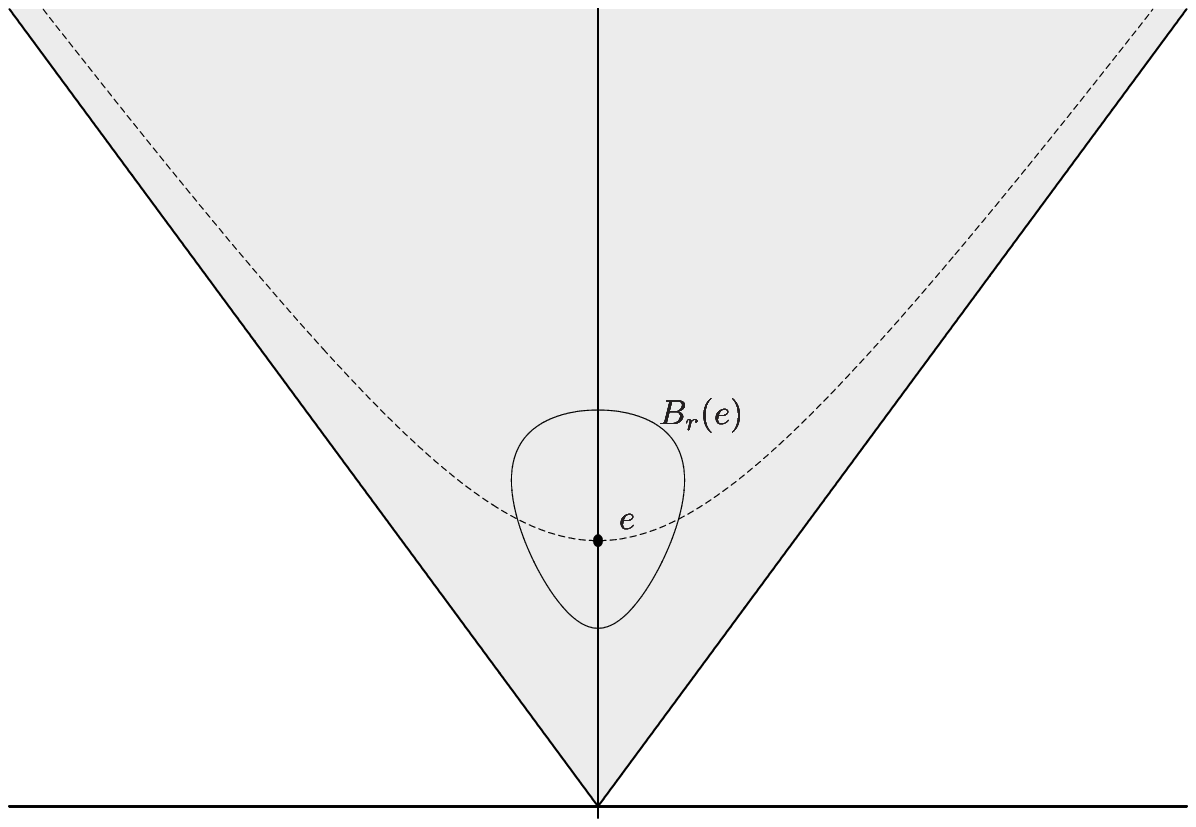}
  }
  \subfigure[Translates of $B_r(e)$]{
    % \tiny
    % \psfrag{x1}[c]{$x_1$}
    % \psfrag{xn}[l]{$x_n$}
    % \psfrag{xcB}[c]{$B_r(w_j)$}
    % \psfrag{xbB}[c]{$B_r(w_{j-1})$}
    % \psfrag{xaB}[c]{$B_r(w_{j-2})$}
    % \psfrag{xdB}[c]{$B_r(w_{j+1})$}
    % \psfrag{xeB}[c]{$B_r(w_{j+2})$}
    \includegraphics[width=0.45\columnwidth]{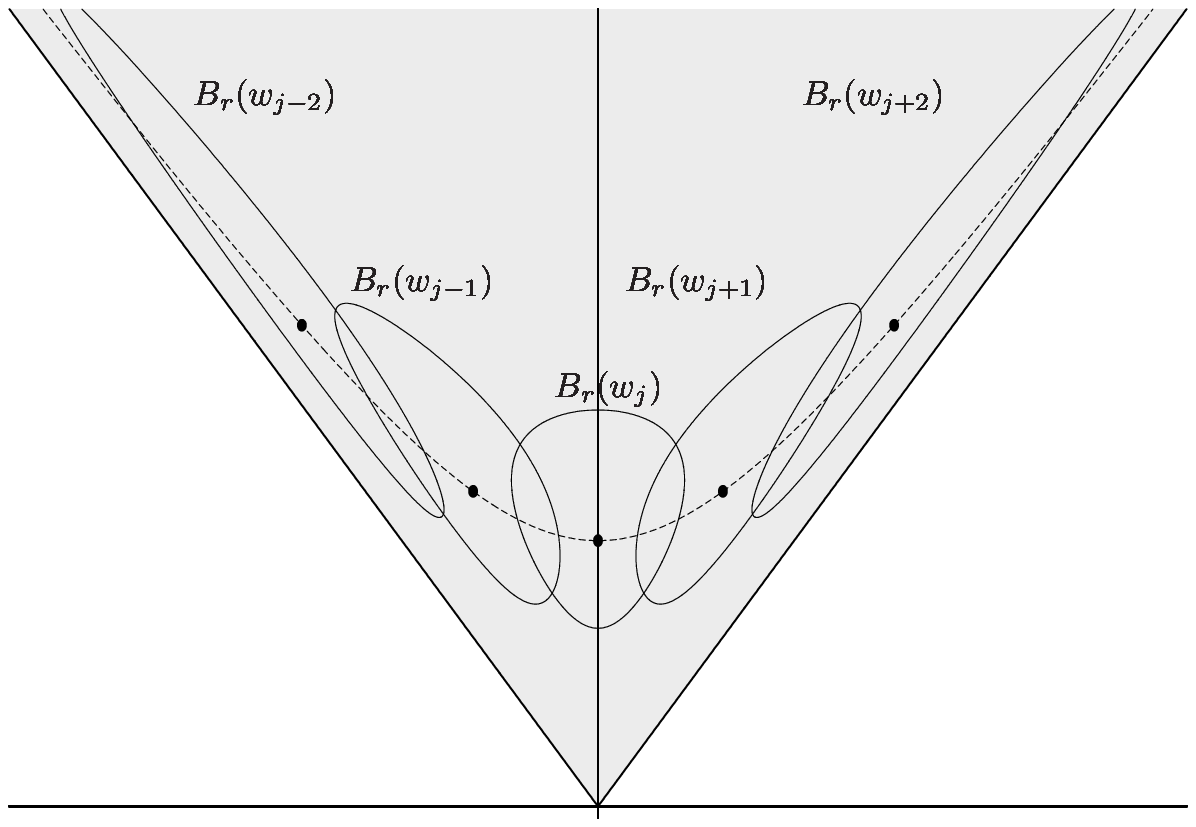}
  }
  
  \subfigure[Cover with translates of $B_r(e)$]{
     \includegraphics[width=0.45\columnwidth]{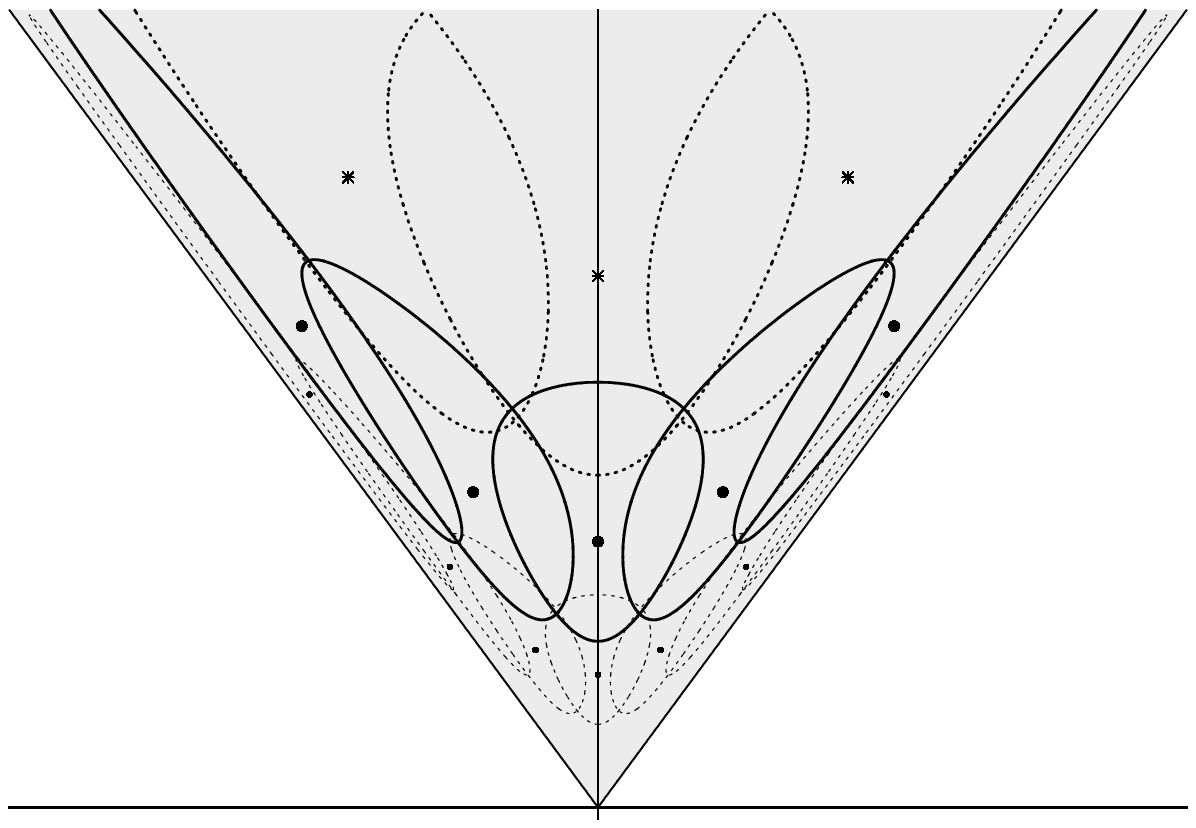}
  }
  \subfigure[Translates of $B_{r/2}(e)$ have no overlap]{
    \includegraphics[width=0.45\columnwidth]{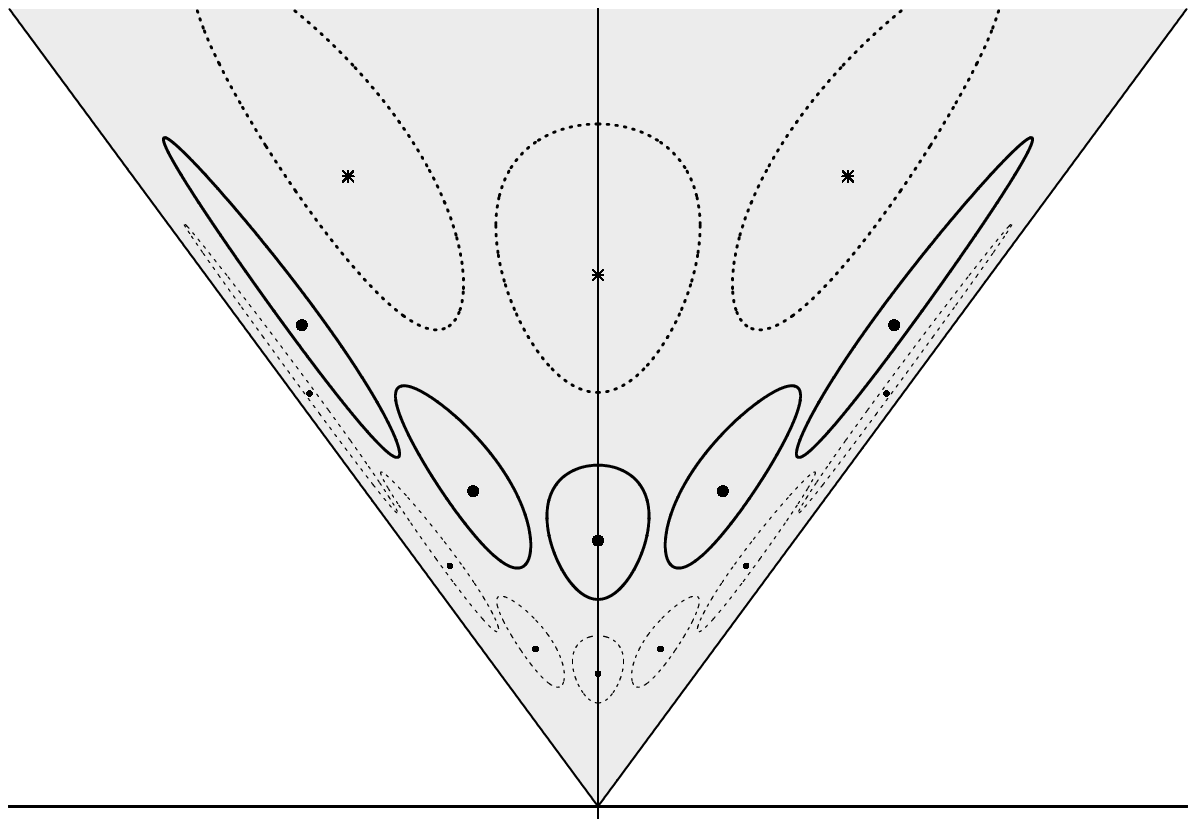}
  }
  \caption{Covering of the cone}  \label{fig:conecover}
\end{figure}
We now construct a smooth partition
of unity subordinate to a cover from Lemma~\ref{lem:covercone}.
Let $0\leq \varphi\leq 1$ be a smooth 
function with support in $B_{2\delta} (e)$ 
such that $\varphi =1$ on $B_{\delta}(e)$. 
Each of the points
$w_j\in \Lambda$ can be written $w_j=\gamma_ja_{t_j}n_{c_j}e$ 
for $g_j = \gamma_j a_{t_j}n_{c_j}\in \mathbb{R}_+AN$
and now we define $\varphi_j(w) = \varphi(g_j^{-1}w)$. Then the 
function $\Phi = \sum_j \varphi_j$ is smooth and 
bounded from above and below (by the finite intersection property), 
and we can finally define the function $\psi_j$ by letting
$\widetilde{\psi}_j = \varphi_j/\Phi$. We then see that
$\widetilde{\psi}_j$ is smooth and with compact support in 
$B_{2\delta}(w_j)$, $\widetilde{\psi}_j =1$ on $B_{\delta/2}(w_j)$
and $\sum_j \widetilde{\psi}_j(w) =1$ for 
all $w\in \Lambda$. Such a partition of unity is called a
Littlewood-Paley decomposition of the cone subordinate
to a Whitney cover.

We note that the convolutions encountered in this section
are distributional convulutions in $\mathbb{R}^n$.
We are now ready to define the Besov spaces
on the light cone as in \cite{Bekolle2004} 
\begin{definition}
  Let $\psi_j$ be a Littlewood-Paley decomposition
  of the cone subordinate to a Whitney cover with lattice
  points $w_j$.
  For $1\leq p,q <\infty$ define the norm
  \begin{equation*}
    \| f\|_{B^{p,q}_s} 
    = \Big( \sum_j \det^{-s}(w_j) \| f*\psi_j\|_p^q \Big)^{1/q}
  \end{equation*}
  then the space $B^{p,q}_s$ consist of the
  $f\in \mathcal{S}_\Lambda'$ for which $\| f\|_{B^{p,q}_s} < \infty$.
\end{definition}
In \cite[Lemma 3.8]{Bekolle2004} it is further proven, that
$B^{p,q}_s$ does not depend (up to norm equivalence) on the
functions $\psi_j$ nor on the Whitney decomposition. We
will use this in the sequel.

\begin{theorem}
  The Besov space $B^{p,q}_{n-s-nq/2}$ corresponds to the coorbit
  $\Co_{\mathcal{S}_\Lambda}^u L^{p,q}_s (G)$ with equivalent norm.
\end{theorem}

\begin{proof}
  First show that 
  $B^{p,q}_{s+nq/2-n} \subseteq \Co_{\mathcal{S}_\Lambda}^u L^{p,q}_s (G)$.
  Assume that $f\in B^{p,q}_{s+nq/2-n}$ and that $\widetilde{\phi}_i$ is 
  a Littlewood-Paley decomposition of the cone with lattice
  points $w_i=g_ie = \gamma_i a_{t_i} n_{c_i}e$. Further assume
  that the sets $g_iV$ cover the cone for an 
  open set $V$ with compact closure. Denote by $U$ the subset of $H$
  given by $U=\{g\in H | ge\in V \}$
  Let $u\in\mathcal{S}_\Lambda$ be a non-zero wavelet for
  which $\widetilde{u}$ has compact support containing the identity. By 
  ${u}_{\gamma a_t n_c}$ denote the function
  $${u}_{\gamma a_t n_c}(x)
  = \gamma^{-n} \overline{{u}((\gamma a_t n_c)^{-1}x)},$$
  then 
  \begin{equation*}
    W_u(f)(\gamma a_t n_c, b) 
    = \gamma^{-n/2} \int f(x)\overline{u((\gamma a_t n_c)^{-1}(x-b))}\, dx
    = \gamma^{n/2} f*u_{\gamma a_t n_c}(b).
  \end{equation*}
  Let the disjoint sets $V_i\subseteq \Lambda$ cover
  $\Lambda$ and satisfy $V_i\subseteq g_i V$. 
  Now choose the subsets $U_i$ to be 
  $U_i = \{ g\in H | ge \in V_i \}$.
  We can then write the $L^{p,q}_s$ norm of the wavelet coefficient as
  \begin{align*}
    \| W_u(f) \|_{L^{p,q}_s} 
    &= \Big( \int_H \gamma^{s+nq/2-n} \| f*u_{\gamma a_t n_c} \|_p^q 
    \frac{d\gamma\,dt\,dc}{\gamma^{n+1}}\Big)^{1/q} \\
    &= \Big( \int_H \gamma^{n-s-nq/2} \| f*u_{\gamma^{-1} a_t n_c} \|_p^q 
    \frac{d\gamma\,dt\,dc}{\gamma^{1}}\Big)^{1/q} \\
    &\leq 
    \Big( \sum_i \int_{U_i} \gamma^{n-s-nq/2} 
    \| f*u_{\gamma^{-1} a_t n_c} \|_p^q 
    \frac{d\gamma\,dt\,dc}{\gamma}\Big)^{1/q} \\
    &\leq C\Big( \sum_i \gamma_i^{n-s-nq/2}  \int_{U_i} 
    \| f*u_{\gamma^{-1} a_t n_c} \|_p^q 
    \frac{d\gamma\,dt\,dc}{\gamma}\Big)^{1/q}, \\
  \end{align*}
  where we have used that $\gamma$ is comparable to $\gamma_i=\det(w_i)$ 
  when $\gamma a_tn_c \in U_i$.
  For any $j$ define 
  $\widetilde{\phi}_{i,j}= \ell_{g_j} \widetilde{\phi}_i$.
  Since $\{ \widetilde{\phi}_i\}_{i}$ is a Littlewood-Paley
  decomposition of the cone the systems
  $\{\widetilde{\phi}_{i,j}\}_{j}$ (with index $j$) and
  $\{\widetilde{\phi}_{i,j}\}_{i}$ (with index $i$) are
  also Littlewood-Paley decompositions of the cone.
  For fixed $i$ we thus
  can write $\| f*u_{\gamma^{-1} a_t n_c}\|_p$ as
  \begin{equation*}
    \| f*u_{\gamma^{-1} a_t n_c}\|_p
    = \Big\| \sum_{j\in J} f*u_{\gamma^{-1} a_t n_c}*\phi_{i,j}\Big\|_p
    \leq \sum_{j\in J} \| f*u_{\gamma^{-1} a_t n_c}*\phi_{i,j}\|_p.
  \end{equation*}
  The index set $J$ in this sum is finite, since both $\widetilde{u}$ and
  $\widetilde{\phi}$ are compactly supported and $w_i$ are well-spread. 
  Further the index set $J$ can be
  chosen large enough that it neither depends on $i$ nor
  on $\gamma a_t n_c \in U_i$.
  The $L^1(\mathbb{R}^n)$-norm 
  of $u_{\gamma^{-1} a_t n_c}$ is uniformly bounded from above, in fact 
  $\| u_{\gamma^{-1} a_t n_c} \|_{L^1(\mathbb{R}^n)} 
  = \| u \|_{L^1(\mathbb{R}^n)}$,
  so we obtain that
  \begin{equation*}
    \| f*u_{\gamma^{-1} a_t n_c}\|_p
    \leq \sum_{j\in J} \| f*\phi_{i,j}\|_p.
  \end{equation*}
  Therefore
  \begin{align*}
    \| W_u(f) \|_{L^{p,q}_s} 
    &\leq C \Big(  \sum_i \gamma_i^{n-s-nq/2}  \int_{U_i} 
    \Big(\sum_{j\in J} \| f*\phi_{i,j} \|_p\Big)^q 
    \frac{d\gamma\,dt\,dc}{\gamma}\Big)^{1/q} \\
    &\leq C \Big(  \sum_i \gamma_i^{n-s-nq/2} 
    \Big(\sum_{j\in J} \| f*\phi_{i,j} \|_p\Big)^q \Big)^{1/q},
  \end{align*}
  where we have used that 
  the $U_i\subseteq g_i U$ have uniformly bounded measure.
%  (since $\frac{d\gamma\,dt\,dc}{\gamma}$ is the invariant measure
%  on $H$). 
  The triangle inequality for the $\ell^q$-norm then yields
  \begin{align*}
    \| W_u(f) \|_{L^{p,q}_s} 
    &\leq C \sum_{j\in J}\Big( \sum_i \gamma_i^{n-s-nq/2} 
    \| f*\phi_{i,j} \|_p^q 
    \Big)^{1/q}.
  \end{align*}
%  When we translate $\widetilde{\phi}_i$ by $g_j$ to get a
%  new Littlewood-Paley decomposition 
%  $\{ \widetilde{\phi}_{i,j} \}_i$ the associated Besov space norm
%  is given by
%  \begin{equation*}
%    \| f\|_{B^{p,q}_s} 
%    = \Big( \sum_i \det^{-s}(g_jw_i) \| f*\psi_{i,j}\|_p^q \Big)^{1/q}
%  \end{equation*}
  Now set 
  $\gamma_{i,j}=\det(g_jw_i)=\gamma_i \gamma_j$, then,
  since the sum over $J$ is finite,  
  each $\gamma_i$ is comparable to
  $\gamma_{i,j}$ for all $j$, and   
  finally we obtain
  \begin{align*}
    \| f \|_{L^{p,q}_s} 
    &\leq C \sum_{j\in J}\Big( \sum_i \gamma_i^{n-s-nq/2} 
    \| f*\phi_{i,j} \|_p^q 
    \Big)^{1/q} \\
    &\leq C \sum_{j\in J} \Big( \sum_i \gamma_{i,j}^{n-s-nq/2} 
    \| f*\phi_{i,j} \|_p^q 
    \Big)^{1/q} \\
    &= C \sum_{j\in J} \Big( \sum_i \det(g_iw_i)^{-(s+nq/2-n)} 
    \| f*\phi_{i,j} \|_p^q 
    \Big)^{1/q}.
  \end{align*}
  Each of the $\{ \phi_{i,j}\}_{i}$ form a Littlewood-Paley 
  decomposition of the cone, so 
  the terms 
  $$%\Big( \sum_i \gamma_{i,j}^{s+nq/2-n} \| f*\phi_{i,j} \|_p^q \Big)^{1/q}
  \Big( \sum_i \det(g_jw_i)^{-(s+nq/2-n)} \| f*\phi_{i,j} \|_p^q
  \Big)^{1/q}
  $$
  are Besov space norms. Each norm is 
  comparable to
  $\| f\|_{B^{p,q}_{s+nq/2-n}}$ by \cite[Lemma 3.8 and expression
  (3.20)]{Bekolle2004}. 
  This shows, that there is a $C>0$ such that
  $$\| W_u(f)\|_{L^{p,q}_s} \leq C \| f \|_{B^{p,q}_{s+nq/2-n}}.$$

  \noindent
  It remains to show that
  $\Co_{\mathcal{S}_\Lambda}^u L^{p,q}_s (G) \subseteq B^{p,q}_{s+nq/2-n}$.
  Let $\widetilde{\phi}$ be the smooth function with support in
  $B_{2\delta}(e)$ used to generate 
  a Littlewood-Paley decomposition. 
  The coorbit spaces are independent of the wavelet $u$, so 
  we choose $u$ and a compact neighbourhood $U\subseteq H$ 
  such that $U\supp(\widetilde{\phi})$ is contained 
  in $\widetilde{u}^{-1}(\{ 1\})$. It holds that
  the $g_iU$'s have finite overlap (the $g_i$'s come from the lattice points
  $w_i = g_ie$ which are well-spread). 
  This means that $\supp(\widetilde{\phi}_i)$
  is contained in $(\widetilde{u}_{\gamma^{-1} a_t n_c})^{-1}(\{ 1\})$ for 
  $\gamma a_t n_c \in g_iU$. Therefore 
  $\widetilde{\phi}_i\widetilde{u}_{\gamma^{-1} a_t n_c} = \widetilde{\phi}_i$ 
  for all $\gamma a_t n_c \in g_i U$. We exploit this to see that
  \begin{align*}
    \|f*\phi_i \|_p^q 
    &=\frac{1}{|U|} \int_{g_iU} \|f*\phi_i \|_p^q 
    \frac{\,d\gamma \,dt\,dc}{\gamma} \\
    &= \frac{1}{|U|} \int_{g_iU} \|f*\phi_i*u_{\gamma^{-1} a_t n_c} \|_p^q 
    \frac{\,d\gamma \,dt\,dc}{\gamma} \\
    &\leq C \int_{g_iU} \|f*u_{\gamma^{-1} a_t n_c} \|_p^q 
    \frac{\,d\gamma \,dt\,dc}{\gamma},
  \end{align*}
  where  $\frac{\,d\gamma \,dt\,dc}{\gamma}$ is the invariant measure
  on the group $H$. In the last step
  we used that $\| \phi_i \|_{L^1(\mathbb{R}^n)}$ is uniformly bounded 
  (see \cite[Proposition 3.2(3)]{Bekolle2004}).
  For $\gamma a_t n_c \in g_iU$ we see that $\gamma$ is comparable to
  $\gamma_i$. Further the sets $g_iU$ overlap a finite amount of times, so
  we obtain the estimate
  \begin{align*}
    \sum_i \gamma_i^{-(s+nq/2-n)} \|f*\phi_i \|_p^q 
    &\leq C \int_H \gamma^{-(s+nq/2-n)}\|f*u_{\gamma^{-1} a_t n_c} \|_p^q 
    \frac{\,d\gamma \,dt\,dc}{\gamma} \\
    &= C \int_H \gamma^{s+nq/2-n}\|f*u_{\gamma a_t n_c} \|_p^q 
    \frac{\,d\gamma \,dt\,dc}{\gamma}.
  \end{align*}
  We use this 
  to find the estimate of the Besov space norm
  \begin{align*}
    \| f\|_{B^{p,q}_{s+nq/2-n}}
    &=
    \Big( \sum_i \det(w_i)^{-(s+nq/2-n)} \|f*\phi_i \|_p^q \Big)^{1/q} \\
    &= \Big( \sum_i \gamma_i^{-(s+nq/2-n)} \|f*\phi_i \|_p^q \Big)^{1/q} \\
    &\leq 
    C \Big( \int_H \gamma^{s+nq/2}\|f*u_{\gamma a_t n_c} \|_p^q 
    \frac{\,d\gamma \,dt\,dc}{\gamma^{n+1}} \Big) \\
    &=
    C \| W_u(f)\|_{L^{p,q}_s}.
  \end{align*}
  
  This proves the equivalence of the norms of the two spaces.
\end{proof}

\begin{remark}
  The wavelet characterization of Besov spaces on forward light cones
  seems to generalize to all symmetric cones. 
  We will deal with this in future work.
\end{remark}

\bibliographystyle{plain}
\bibliography{coorbit}
\end{document}